\documentclass[12pt]{article}
\usepackage{graphicx} % Required for inserting images
\usepackage{color}
\usepackage{enumerate}
\usepackage{latexsym}
\usepackage[T1]{fontenc}
\usepackage{float}
\usepackage{amssymb,amscd,amsmath,amsfonts,amsthm}
\usepackage{multicol}
\usepackage{tabularx,environ,array}
\usepackage{tikz}
\usepackage{comment}
\usepackage{hyperref}
\usepackage{boxedminipage}
\usepgflibrary{shapes.geometric}

\hypersetup{colorlinks=true, linkcolor=blue, citecolor=blue, urlcolor=blue}

\newcommand{\chimu}{\chi_{\mu}}
\newcommand{\chimut}{\chi_{\mu_t}}
\newcommand{\chimuo}{\chi_{\mu_o}}
\newcommand{\chimud}{\chi_{\mu_d}}

\newcommand{\mud}{\mu_{\mathrm{d}}}
\newcommand{\mut}{\mu_{\mathrm{t}}}
\newcommand{\muo}{\mu_{\mathrm{o}}}

\newcommand{\sa}{\mathrm{sa}}
\newcommand{\cp}{\,\square\,}

\newtheorem{theorem}{Theorem}

\newtheorem{corollary}[theorem]{Corollary}

\newtheorem{lemma}[theorem]{Lemma}

\newtheorem{proposition}[theorem]{Proposition}
\newtheorem{remark}[theorem]{Remark}

\begin{document}

\title{A variety of the mutual-visibility coloring problem for graphs}

\author{Saneesh Babu$^{a,}$\thanks{\texttt{saneeshbabu@cusat.ac.in}}
\and Marko Jakovac$^{b,c,}$\thanks{\texttt{marko.jakovac@um.si}}
\and Dorota Kuziak$^{d,}$\thanks{\texttt{dorota.kuziak@uca.es}}
\and Aparna Lakshmanan S.$^{a,}$\thanks{\texttt{aparnals@cusat.ac.in}}
\and Ismael G. Yero$^{e,}$\thanks{\texttt{ismael.gonzalez@uca.es}}
}
\maketitle

\begin{center}
\vspace*{-0.7cm}
$^a$ Department of Mathematics, Cochin University of Science and Technology, India\\

\medskip
$^b$ Faculty of Natural Sciences and Mathematics, University of Maribor, Slovenia\\

\medskip
$^c$ Institute of Mathematics, Physics and Mechanics, Ljubljana, Slovenia\\

\medskip
$^d$ Departamento de Estad\'istica e IO, Algeciras Campus, Universidad de C\'adiz, Spain \\

\medskip
$^e$ Departamento de Matem\'{a}ticas, Universidad de C\'adiz, Algeciras Campus, Spain 
\end{center}

\begin{abstract}
This paper explores variations of vertex-coloring problems defined on graph visibility properties. It introduces and studies the dual, outer, and total mutual-visibility chromatic numbers, which partition the vertex set of a graph into color classes that preserve specific mutual-visibility conditions called dual, outer or total. The work provides structural conditions under which these chromatic parameters are finite or infinite, and establishes that deciding whether a graph admits a dual, outer, or total mutual-visibility coloring using a given number of colors is NP-complete, even when restricted to two colors.  Exact formulas and tight bounds for these chromatic parameters are established across several fundamental graph classes. 

For block graphs, complete characterizations are provided for the outer and dual mutual-visibility chromatic numbers based on structural invariants such as cut vertices and specific forbidden subgraph structures. On Hamming graphs, the dual and total mutual-visibility chromatic numbers are shown to equal the smaller dimension of the factors, while the outer mutual-visibility chromatic number is proven to equal the star arboricity of a corresponding complete bipartite graph. Finally, the paper examines strong grid graphs, determining exact values for their total, outer, and dual mutual-visibility chromatic numbers. These results demonstrate how the parameter behaviors range from finite constants to infinity depending on the grid dimensions.
\end{abstract}

\textbf{Math.\ Subj.\ Class. 2020:} 05C12, 05C69


%%%%%%%%%%%%%%%%%%%%%%%%%%%%%%%%%%%%%%%%%%%%%%%%%%%

\section{Introduction} 

The mutual-visibility problem in graphs was introduced just three years ago in \cite{DiS} as a theoretical framework for some computer science problem related to the navigation of robots in networks avoiding collisions. Surprisingly, the topic has attracted the attention of several researchers after the seminal article appeared. A proof of this is observed since the reader can already find more than 60 research works in the last three years which are focused on providing new contributions on visibility related problems. 

A remarkable fact that might have contributed to this interest relates to the close relationships that have appeared between the mutual-visibility problem and other classical combinatorial topics like for instance the Zarankiewicz problem, some Tur\'an-type problems in hypergraphs, or some problems from Ramsey theory. To see this, we suggest for some examples the recent works \cite{BPSY,BKT,CDK,CDKY2,KKVY}.

A typical research direction that has emerged is the one considering variations of the original problem. That is, some researchers have modified, generalized or combined the mutual-visibility property so that new research lines have appeared. Among them, we remark \cite{CDDH}, where a variety of mutual-visibility properties was introduced and studied. The formal concepts are next defined.

From now on, we consider $G = (V(G), E(G))$ as a connected, non-trivial, undirected simple graph of order $n(G)=|V(G)|$ and size $m(G)=|E(G)|$. Given a set of vertices $M\subseteq V(G)$ and two vertices $x,y\in V(G)$, it is said that $x,y$ are $M$-\textit{visible}, if there is a shortest path $P$ between $x$ and $y$ (shortly, an $x,y$-path $P$) such that $V(P)\cap M\subseteq \{x,y\}$. Roughly speaking, this means that all internal vertices of $P$ are not included in $M$. Having this notion in mind, the following graph structures are known.
\begin{itemize}
    \item $M$ is a \textit{mutual-visibility set} if each two $x,y\in M$ are $M$-visible (from \cite{DiS}).
    \item $M$ is a \textit{dual mutual-visibility set} if each two $x,y\in M$ and each two $x',y'\notin M$ are $M$-visible (from \cite{CDDH}).
    \item $M$ is an \textit{outer mutual-visibility set} if each two $x,y\in M$ and each two $x',y'\in V(G)$ such that $x'\in M$ and $y'\notin M$ are $M$-visible (from \cite{CDDH}).
    \item $M$ is a \textit{total mutual-visibility set} if each two $x,y\in V(G)$ are $M$-visible (from \cite{CDKY}).
\end{itemize}
With these types of sets in mind, the cardinality of a (resp.) largest possible set from the above ones are known as: the \textit{mutual-visibility number} $\mu(G)$, the \textit{dual mutual-visibility number} $\mud(G)$, the \textit{outer mutual-visibility number} $\muo(G)$ and the \textit{total mutual-visibility number} $\mut(G)$, resp.
%A given set $S\subseteq V(G)$ is called a $\pi$-\textit{set} of $G$ with $\pi\in \{\mu,\mut,\mud,\muo\}$ (resp.), if $|S|$ equals $\mu(G)$, $\mut(G)$, $\mud(G)$ or $\mud(G)$ (resp.). 

Another closely related problem to the ones above is the general position problem and its variety. The general position problem is a stronger concept than the mutual-visibility, although it indeed was introduced much earlier than the mutual-visibility problem. For more information on this topic, we suggest the survey \cite{CKT-survey}, as well as, the articles \cite{RKLT} and \cite{TK-25} where the variety was described.

On the other hand, as it usually happens, the mutual-visibility and general position concepts have been related to classical graph coloring problems. For instance, in \cite{KKVY}, an investigation considering coloring the vertices of a graph in such a way that each color class represents a mutual-visibility set was presented. This investigation was right away followed by the work \cite{Axenovich-pub} (or \cite{Axenovich} for a previous version), which studied such coloring in hypercubes. A similar study that considered coloring the vertices of a graph into general position sets was presented in \cite{CDiSHTT}. In addition, a related coloring in which an independence property in the mutual-visibility sets of the color classes of such coloring was developed in \cite{BPSY}. Another variation which insists that the distance between every pair of vertices in the mutual-visibility set can be at most $k$, termed $k$-distance mutual-visibility set, and the corresponding coloring is introduced and studied in \cite{BBSL}. In concordance with these works, and with the notions regarding the variety of mutual-visibility problems already known, we are hence focused on describing and further studying a variety of mutual-visibility coloring problems for graphs. We next formalize our statements.

Assuming from now on that $[n]=\{1,\dots,n\}$ for any positive integer $n\ge1$, we consider a mapping $c:V(G)\rightarrow [k]$ for a graph $G$. Such mapping $c$ is: 
\begin{itemize}
    \item a {\em mutual-visibility $k$-coloring} of $G$ if every color class $c^{-1}(i)$ is a mutual-visibility set of $G$ for every $i\in [k]$ (we remark that this notion is already known from \cite{KKVY} and we include it here for completeness in the variety); or
    \item a {\em dual mutual-visibility $k$-coloring} of $G$ if every color class $c^{-1}(i)$ is a dual mutual-visibility set of $G$ for every $i\in [k]$; or
    \item an {\em outer mutual-visibility $k$-coloring} of $G$ if every color class $c^{-1}(i)$ is an outer mutual-visibility set of $G$ for every $i\in [k]$; or
    \item a {\em total mutual-visibility $k$-coloring} of $G$ if every color class $c^{-1}(i)$ is a total mutual-visibility set of $G$ for every $i\in [k]$.
\end{itemize}
Notice that such mapping $c$ generates a partition $P=\{P_1\dots,P_k\}$ of the vertex set of $G$, which are the color classes according to $c$. From now, we shall refer to such $P$ as (resp. a dual, an outer, or a total) a mutual-visibility $k$-coloring of $G$ and we shall also say that $P$ is a $\pi$-coloring of $G$ with $\pi\in \{\chimu,\chimud,\chimuo,\chimut\}$, in connection with the parameters that are defined next. The smallest integer $k$ such that the graph $G$ admits (resp. a dual, an outer, or a total) a mutual-visibility $k$-coloring is the (resp. {\em dual, outer}, or {\em total$)$ mutual-visibility chromatic number} of $G$, denoted by $\chimu(G)$ (resp. $\chimud(G)$, $\chimuo(G)$, or $\chimut(G)$). 

\subsection{Plan of the article}

Section \ref{sec:prelim} establishes general bounds for the parameters of our work. For example, we present a characterization of the graphs $G$ such that $\chi_{\mu_t}(G) = \infty$, as those graphs containing a vertex which is the center of a convex $P_3$. With respect to a similar characterization for the graphs $G$ with $\chi_{\mu_d}(G) = \infty$, we have only shown that if $G$ contains a vertex that is the center in a convex $P_5$ or in a convex $K_{1,3}$, then $\chi_{\mu_d}(G) = \infty$. However, a complete characterization is still missing. Also, we introduce mutually maximally distant classes to characterize extreme cases for outer mutual-visibility, showing that $\chi_{\mu_o}(G) = 1$ is satisfied only for complete graphs; and that $\chi_{\mu_o}(G) = n-1$ holds only for paths. Section \ref{sec:complex} considers the complexity of the decision problems regarding total, dual, and outer mutual-visibility colorings, which are proved to be NP-complete, even when restricted to two colors ($k=2$) via a polynomial-time reduction from NAE3-SAT. Section \ref{sec:block} shows that $\chi_{\mu_t}(G) = \infty$ follows for any block graph with at least two blocks, and that $\chi_{\mu_o}(G) = \vert{}\Omega\vert{} + 1$ for any block graph, where $\Omega$ is the set of cut vertices. In addition, in this section, we find the exact value of $\chi_{\mu_d}(G)$ as 1, $\vert{}\Omega\vert{}$, $\vert{}\Omega\vert{}+1$ or infinity, for various classes of block graphs $G$, where $\Omega$ denotes the set of cut vertices of $G$. In Section \ref{sec:Hamming} we determine that $\chi_{\mu_t}(K_m \square K_n) = \chi_{\mu_d}(K_m \square K_n) = \min\{m,n\}$, and show an equivalence between the outer parameter and another graph parameter called star arboricity, that is, $\chi_{\mu_o}(K_m \square K_n) = sa(K_{m,n})$. From this relationship, we derive exact values for squared cases, including $\chi_{\mu_o}(K_n \square K_n) = \lceil n/2 \rceil + 2$ for $n \ge 7$. In Section \ref{sec:grids}, we study the strong grids $P_r \boxtimes P_t$ ($r \ge t \ge 2$), showing that $\chi_{\mu_t}(P_r \boxtimes P_t)$ is 1 for $t=2, r=2$; 2 for $t=2, r \ge 3$; and $\infty$ for $t \ge 3$. We also establish that $\chi_{\mu_o}(P_r \boxtimes P_t) = t - 1$ for $r \ge t \ge 3$. In addition, we compute $\chi_{\mu_d}(P_r \boxtimes P_t)$, proving it equals 1 for $t=2, r=2$; 2 for $t=3, r=3$; $t$ for $2 \le t \le 4$; and $\infty$ for $t \ge 5$. We end our contributions in Section \ref{sec:conclu}, where we remark a collection of open questions that might be worth considering as a continuation of this investigation.

\section{Preliminaries}\label{sec:prelim}

Based on the definitions of all the concepts of our work, one can observe that the following inequalities are satisfied for any graph $G$.
\begin{equation}
\label{eq:chimu-d-t}
\chimu(G)\le\chimud(G)\le \chimut(G).
\end{equation}
\begin{equation}
\label{eq:chimu-o-t}
\chimu(G)\le\chimuo(G)\le \chimut(G).
\end{equation}

On the other hand, it can be readily observed that since each color class of any (dual, outer, or total) mutual-visibility $k$-coloring is (a dual, an outer, or a total) mutual-visibility set, for any graph $G$, it holds that 
\begin{equation}\label{eq:chimus-lower-bounds}
    \chimu(G)\ge \left\lceil\frac{n(G)}{\mu(G)}\right\rceil,\quad
    \chimud(G)\ge \left\lceil\frac{n(G)}{\mud(G)}\right\rceil,\quad
    \chimuo(G)\ge \left\lceil\frac{n(G)}{\muo(G)}\right\rceil,\quad
    \chimut(G)\ge \left\lceil\frac{n(G)}{\mut(G)}\right\rceil.
\end{equation}

It is already known from \cite{CDKY} that there might be graphs $G$ for which $\mud(G)=0$ or $\mut(G)=0$ (those graphs having no (dual or total) mutual-visibility sets). In this sense, and being consequent with the inequalities \eqref{eq:chimus-lower-bounds}, it makes sense to take the agreements that there might be graphs $G$ for which $\chimud(G)=\infty$ or $\chimut(G)=\infty$.

The class of graphs $G$ satisfying that $\mut(G)=0$ was already characterized in \cite{TK} as follows. On the other hand, some partial results on the graphs $G$ satisfying that $\mud(G)=0$ were presented in \cite{CDKY}.

\begin{theorem}{\em \cite{TK}}
\label{th:mut-0-charact}
Let $G$ be a graph. Then $\mut(G)=0$ if and only if every vertex of $G$ is the center of a convex $P_3$.
\end{theorem}

\begin{theorem}{\em \cite{CDKY}}
\label{th:mud-0-partial}
Let $G$ be a graph.  If every two adjacent vertices of $G$ are the center of a convex $P_4$, then $\mud(G) = 0$.
\end{theorem}

Some other examples of graphs satisfying $\mud(G) = 0$ and not satisfying the statement of Theorem \ref{th:mud-0-partial} were given in  \cite{CDKY} as well. Notice that the fact that $\mud(G)=0$ or $\mut(G)=0$ respectively implies that $\chimud(G)=\infty$ or $\chimut(G)=\infty$. However, the opposite is not true, although it is somehow related. We now center our attention into studying the graphs $G$ for which $\chimud(G)=\infty$ or $\chimut(G)=\infty$. 

\begin{theorem}
\label{th_chimut-inf}
Let $G$ be a graph. Then $\chimut(G)=\infty$ if and only if there exists at least a vertex in $V(G)$ which is the center of a convex $P_3$.
\end{theorem}

\begin{proof}
$(\Leftarrow)$ Clearly, if $v$ is the center of a convex $P_3$ in $G$, then $v$ cannot belong to any total mutual-visibility set of $G$. Thus, we cannot form a total mutual-visibility $k$-coloring in $G$, which means $\chimut(G)=\infty$.

$(\Rightarrow)$ Assume next that $\chimut(G)=\infty$. Namely, any partition $\Pi'=\{P'_1,\dots,P'_r\}$ of $V(G)$ does not induce a total mutual-visibility $r$-coloring of $G$. In particular, the partition $\Pi=\{P_1,\dots,P_{n(G)}\}$ of $V(G)$ such that $|P_i|=1$ for each $i\in [n(G)]$ satisfies such property. Let $P_i=\{v\}$ be such that $P_i$ is not a total mutual-visibility set. Notice first that $v$ must have degree at least $2$, for otherwise, it is indeed a total mutual-visibility set. This means that there are two vertices $x,y\in V(G)$ such that $x,y$ are not $P_i$-visible (note that neither $x$ nor $y$ can be $v$), i.e., $v$ is an internal vertex of every shortest $x,y$-path. Let $Q$ be one such path, and let $x',y'\in N_G(v)$ such that $x'$ belongs to the shortest $v,x$-path and $y'$ belongs to the shortest $v,y$-path. Clearly, it might happen $x=x'$ or $y=y'$. One now readily observes that $x'vy'$ must be a convex $P_3$ (for otherwise $x,y$ will be $P_i$-visible), and $v$ is the center of such a convex $P_3$.
\end{proof}

In the case of total mutual-visibility coloring of a graph $G$, $\chimut(G)=\infty$ is equivalent to having a vertex $v \in V(G)$ which does not belong to any total mutual-visibility set, since the fact that a set $S$ is a total mutual-visibility set implies that every subset of $S$ is also a total mutual-visibility set. However, the case of dual mutual-visibility coloring is different, as each subset of a dual mutual-visibility set is not necessarily a dual-mutual visibility set as well. So, there are graphs for which every vertex belongs to some dual-mutual visibility set, but still the vertex set cannot be partitioned into dual-mutual visibility sets. For instance, one may consider the cycle $C_5$. This makes the characterization of graphs for which $\chimud(G)=\infty$ more challenging. In the following theorem, we have a sufficient condition for a graph $G$ to have $\chimud(G)=\infty$, but the complete characterization of such graphs is still open.

\begin{theorem}
\label{th_chimud-inf}
%Let $G$ be a graph. Then $\chimud(G)=\infty$ if and only if there exist at least a vertex $v\in V(G)$ which is the central vertex in a convex $P_5$ or in a convex  $K_{1,3}$.
Let $G$ be a graph. If there exist a vertex $v\in V(G)$ which is the central vertex in a convex $P_5$ or in a convex  $K_{1,3}$, then $\chimud(G)=\infty$.
\end{theorem}

\begin{proof} 
First, let $v$ be the central vertex of a convex $K_{1,3}$ with at least three neighbors not pairwise adjacent between them, and suppose $v\in S$, where $S$ is a dual mutual-visibility set of $G$. Hence, at most one of its three neighbors belongs to $S$, for otherwise the mutual-visibility property of vertices in $S$ is violated. Thus, at least two neighbors of $v$ are not in $S$, but then, the  mutual-visibility property of vertices not in $S$ is also violated. Both contradictions lead to confirm that $v$ does not belong to any dual mutual-visibility set of $G$, and so, it holds $\chimud(G)=\infty$ in this situation.

Now, let $v_1v_2v_3v_4v_5$ be a convex $P_5$. Suppose, for a contradiction, that $v_3$ belongs to a dual mutual-visibility set $M$. By the dual mutual-visibility property of $M$, $v_2$ and $v_4$ neither both belong to $M$ nor both lie outside $M$. Without loss of generality, assume $v_2 \in M$. Consider the vertex $v_1$.  
\begin{itemize}
    \item If $v_1 \in M$, then $v_1,v_3\in M$ are not $M$-visible.
    \item If $v_1\notin M$, then $v_1,v_4\notin M$ are not $M$- visible.
\end{itemize}
A symmetric argument holds if $v_4\in M$. In all cases, we reach a contradiction. Therefore, $v_3$ cannot belong to any dual mutual-visibility set of $G$, and again $\chimud(G)=\infty$.
\end{proof}

Notice that there are graphs $G$ for which every vertex $v\in V(G)$ is neither the central vertex in a convex $P_5$ nor in a convex  $K_{1,3}$,  however, they still have vertices which are not part of any dual mutual-visibility set. An example of this is for instance the wheel graph on $7$ vertices $C_6 \vee K_1$, for which the central vertex (from $K_1$) cannot be present in any dual mutual-visibility set. 

\medskip
In contrast to the total and dual mutual-visibility numbers (that can be equal to $0$), the outer mutual-visibility number of a non-trivial graph $G$ is always larger than or equal to $2$, since any two diametral vertices of $G$ form an outer mutual-visibility set. Also, any single vertex of $G$ forms an outer mutual-visibility set of $G$. This, together with \eqref{eq:chimus-lower-bounds} and the fact that $\muo(G)\le n$, allow to conclude that for any graph $G$ of order $n$,
\begin{equation}\label{eq-triv-bounds-outer}
1\le \chimuo(G)\le n-1.   
\end{equation}

We next focus on characterizing the classes of graphs attaining the bounds above.
To this end, we need some extra terminology and notation. A vertex $u\in V(G)$ is \textit{maximally distant} from $v\in V(G)$, if $d_G(v,x)\le d_G(v,u)$ for every $x\in N_G(u)$. Two vertices $u,v\in V(G)$ are \textit{mutually maximally distant} (MMD) if $u$ is maximally distant from $v$ and $v$ is maximally distant from $u$. Notice that every graph contains at least a pair of mutually maximally distant vertices (for instance any two diametral vertices are mutually maximally distant). The \textit{boundary} $\partial(G)$ of $G$ represents the set of vertices $u$ for which there exists $v\in V(G)$ such that $u,v$ are mutually maximally distant. Note that we can always find a partition $\{Q_1,\dots,Q_r\}$ of $\partial(G)$ of a smallest possible cardinality with $|Q_j|\ge 1$ for every $j\in [r]$, and, if there is a $Q_i\in \partial(G)$ such that $|Q_i|\ge 2$, then any two vertices $x,y\in Q_i$ are mutually maximally distant in $G$. We note first that in every possible partition of $\partial(G)$ there is at least one set $Q_\ell$ with $|Q_\ell|\ge 2$. Moreover, we observe that such a partition can be frequently made in such a way that $|Q_i|\ge 2$ for every $i\in [r]$. However, this is not always the case. For instance, if we consider an odd cycle, say $C_5: v_1v_2v_3v_4v_5v_1$, then $\partial(C_5) = V(C_5)$, and we cannot partition $V(C_5)$ into mutually maximally distant sets each of which is of cardinality at least 2. Here the MMD sets are $\{v_1,v_3\}, \{v_1,v_4\}, \{v_2,v_4\}, \{v_2,v_5\}$ and $ \{v_3,v_5\}$. In such situation, $\partial(C_5)$ can be partitioned into MMD sets like $\{v_1,v_3\}, \{v_2,v_4\}, \{v_5\}$, where one of the sets is a singleton. From now on, we shall call such sets $Q_i$ as the mutually maximally distant classes of $G$. For some examples, notice that $\partial(K_n)=V(K_n)$ and there is only one mutually maximally distant class, which is formed by the whole set of vertices of $K_n$. Also, if we consider a cycle $C_{2n}$, then $\partial(C_{2n})=V(C_{2n})$ and its $n$ mutually maximally distant classes are formed by pairs of diametral vertices in $C_{2n}$. In a tree $T$, we have that $\partial(T)$ is formed by the set of its leaves and they form its unique mutually maximally distant class. 

The notion of MMD vertices was used in several articles concerning computing the strong metric dimension of graphs. To this end, the concept of strong resolving graph was needed. That is, given a graph $G$, the strong resolving graph $G_{SR}$ of $G$ has vertex set $\partial(G)$ and two vertices $x,y$ are adjacent in $G_{SR}$ if $x,y$ are MMD in $G$. To see more information about such structure, we suggest the survey article \cite{SRG}. This survey considers (among other things) the realization and the characterization of graphs as strong resolving graphs. For example, the following result is proved.

\begin{lemma}{\em \cite{SRG}}
\label{lem:SRG}
Let $G$ be a graph of order $n$. Then,
$G_{SR}\cong K_{1,r}$ for some $r\ge 1$ if and only if $G\cong P_n$ and $r=1$.
\end{lemma}

\begin{lemma}
\label{lem:MMD-class}
Let $G$ be any graph. Then each set of pairwise MMD vertices forms an outer mutual-visibility set of $G$. In particular, each mutually maximally distant class of $G$ of cardinality at least two is an outer mutual-visibility set of $G$.
\end{lemma}

\begin{proof}
Let $Q$ be a set of pairwise MMD vertices of $G$ with at least two vertices $u,v\in Q$. Let $x,y\in V(G)$ such that $|Q\cap \{x,y\}|\ge 1$. Assume w.l.o.g. that $x\in Q$ and suppose that $x,y$ are not $Q$-visible. Hence, every shortest $x,y$-path in $G$ contains at least one internal vertex of $Q$. In this sense, let $P$ be a shortest $x,y$-path and let $z\in Q\cap V(P)$ (notice that $z$ can neither be $x$ nor $y$). Thus, there exists $z'\in N_G(z)\cap V(P)$ such that $d_G(x,z')=d_G(x,z)+1$ (note that such $z'$ can be $y$). However, this means that $x,z$ are not mutually maximally distant in $G$, which is a contradiction since both $x,z\in Q$. Therefore, $x,y$ are $Q$-visible, and so, $Q$ is an outer mutual-visibility set of $G$ as desired. If in particular, $Q$ is a mutually maximally distant class, then a similar argument also applies.
\end{proof}

Notice that the opposite statement in the lemma above is not true. For example, in $C_4$, any two adjacent vertices form an outer mutual-visibility set, but they are not MMD. The lemma above, together with the fact that every vertex of a graph $G$ is an outer mutual-visibility set of $G$, lead to the following corollary.

\begin{corollary}
\label{cor:outer-MMD}
Let $G$ be any graph of order $n$ and let $\{Q_1,\dots,Q_r\}$ be the mutually maximally distant classes of $G$ of cardinality at least two. Then, $\muo(G)\ge \max\{|Q_i|\,:\,i\in [r]\}$ and $\chimuo(G)\le r+n-\sum_{i=1}^r|Q_i|$.
\end{corollary}

The next results give the classes of graphs satisfying the equality in the bounds of \eqref{eq-triv-bounds-outer}.

\begin{proposition}
Let $G$ be a graph of order $n$. Then,
\begin{enumerate}
    \item[{\em (i)}] $\chimuo(G)=1$ if and only if $G$ is a complete graph $K_n$, and
    \item[{\em (ii)}] $\chimuo(G)=n-1$ if and only if $G$ is the path graph $P_n$.
\end{enumerate}
\end{proposition}

\begin{proof}
(i) Clearly, if $G$ is $K_n$, then the whole vertex set of $K_n$ is an outer mutual-visibility set, and so, $\chimuo(G)=1$. On the other hand, if $G$ is not complete, then there are at least two vertices which are not adjacent and they cannot be together in the unique set of an outer mutual-visibility $1$-coloring of $G$. Therefore, $\chimuo(G)\ge 2$, and we are done.

(ii) If $G$ is $P_n=v_1v_2\dots v_n$, then clearly $\{\{v_1,v_n\},\{v_2\},\dots,\{v_{n-1}\}\}$ defines an outer mutual-visibility $(n-1)$-coloring of $P_n$. Also, if $P=\{P_1,\dots, P_k\}$ is an outer mutual-visibility $k$-coloring of $P_n$, then there can be only one class with two vertices (the leaves of $P_n$) and any other class must have only one vertex. Notice that any set of at least two vertices in which at least one of them is not a leaf vertex of $P_n$ is not an outer mutual-visibility set of $P_n$. Thus, $\chimuo(P_n)\ge n-1$, which leads to the required equality.

On the other hand, assume that $\chimuo(G)=n-1$. Since $G$ has $n$ vertices, by the pigeon hole principle, there must be exactly one color class of cardinality $2$, say $A=\{u,v\}$, and every other color class ($n-2$ of them) is of cardinality $1$. Suppose there are two vertices $x,y\notin A$, such that they are MMD in $G$. By Lemma \ref{lem:MMD-class}, $\{x,y\}$ forms an outer mutual-visibility set of $G$, which in turn, leads to $\chimuo(G)\le n-2$, since $\{x,y\}$ together with $A$, and the remaining $n-4$ singleton sets of $G$, will form an outer mutual-visibility $(n-2)$-coloring of $G$, and this is a contradiction. Thus, each pair of vertices in $V(G)\setminus A$ are not MMD in $G$. Consider now any vertex $z\notin A$. If $z$ is MMD with both vertices $u,v\in A$, then $A\cup \{z\}$ forms  an outer mutual-visibility set of $G$, by Lemma \ref{lem:MMD-class}. This also leads to $\chimuo(G)\le n-2$, which is not possible. Thus, the vertex $z\notin A$ might be MMD with at most one vertex from $A$, say $u$. If there is another vertex $z'\notin A\cup \{z\}$ such that $z',v$ are MMD, then we obtain again a contradiction (by using Lemma \ref{lem:MMD-class} once more), since then $\{z,u\}$, $\{z',v\}$ and the remaining $n-4$ vertices (as singleton classes) form an outer mutual-visibility $(n-2)$-coloring of $G$. Consequently, we obtain that, if there are some vertices not in $A$, which are MMD with one vertex from $A$, then all such vertices are MMD with exactly the same vertex from $A$, say $u$, and the vertex $v\in A$ has no other vertex, other than $u$, which is MMD with it. 

In conclusion, the arguments above lead to observe that $u$ is MMD with $v$, and possibly with some other vertices not in $A$, and that there are no other vertices in $G$ that are MMD between them. Having this in mind, we observe that the strong resolving graph $G_{SR}$ (previously defined) is isomorphic to the graph $K_{1,r}$ for some $r\ge 1$ (a star with $r$ leaves with $u$ as the center of the star). Now, from Lemma \ref{lem:SRG}, it holds that $G_{SR}\cong K_{1,r}$ for some $r\ge 1$ if and only if $G\cong P_n$ and $r=1$. Consequently, we have finally obtained that $G$ must be a path.
%Assume now that $\{Q_1,\dots,Q_r\}$ with $r\ge 1$, form the MMD classes of $G$ of cardinality at least $2$. If $r\ge 2$, then by Lemma \ref{lem:MMD-class}, we deduce that $\chimuo(G)\le n-2$ (at least two MMD classes of cardinality at least $2$ and at most $n-4$ singleton classes), a contradiction again. Thus, it holds $r=1$. Moreover, since there must be always at least one pair of MMD vertices in every graph, it must precisely hold that $Q_1=A$. Thus, we deduce that $\partial(G)$ contains only one class of mutually maximally distant vertices, that must have cardinality two, and the remaining vertices of $G$ do not belong to $\partial(G)$. This means that $G$ must be a path.
\end{proof}

We end this section with a result concerning the relationship between any of the mutual-visibility chromatic parameters of a graph in our variety and that of any convex subgraph of such graph. Such a result will be implicitly used in several arguments of our work, although not explicitly mentioned in general.

\begin{remark}\label{lem:convex}
Let $H$ be a convex subgraph of a graph $G$, and let $\chi_\pi$ denote any one of $\chimu$, $\chimuo$, $\chimud$ or $\chimut$. Then $\chi_\pi(G) \geq \chi_\pi(H)$
\end{remark}

\begin{proof}
Let $c: V(G) \to [k]$ where the set of color classes $c^{-1}(i)$ form a $\pi$-coloring of $G$. Assume $H$ is a convex subgraph of $G$.

Consider the restriction of $c$ onto $H$, i.e., $c|_H : V(H) \to [k']$ for some $k'\le k$. The color classes of $c|_H$ are $c^{-1}(i) \cap V(H)$ for $i\in [k']$, and moreover, notice that $c^{-1}(i)$ might be empty for some values of $i$. Such empty classes are removed from $c|_H$. Since $H$ is convex, every shortest path in $G$ between vertices of $H$ is entirely contained in $H$. Thus, visibility properties in $G$, by shortest paths are preserved when restricting to $H$. Consequently, the set of color classes of $c|_H$ form a $\pi$-coloring of $H$ of cardinality $k'\le k$.
%Moreover, by Lemma~\ref{lem:convex}, subsets of $\pi$-sets remain $\pi$-sets. Hence, each color class of $c|_H$ is a $\pi$-set in $H$, and $c|_H$ is a proper $\pi$-coloring of $H$ using at most $k$ colors. It follows that $\chi_\pi(H) \le k = \chi_\pi(G)$.
\end{proof}

\section{Computational Complexity}\label{sec:complex}

As usual in combinatorial investigations, one line that focuses the attention of several researchers is centered into the computational properties of their related decision or optimization problems. Regarding our topic, the following results are known to be NP-complete.
\begin{itemize}
    \item \cite{CDiSHTT}: Deciding whether there is a vertex coloring of a graph whose color classes form general position sets.
    \item \cite{BDiSL}: Deciding whether there is a vertex coloring of a graph whose color classes form mutual-visibility sets.
    \item \cite{BPSY}: Deciding whether there is a vertex coloring of a graph whose color classes form independent mutual-visibility sets.
\end{itemize}

In this sense, since we deal with the remaining coloring related parameters of the variety of mutual-visibility problems, in this section, we determine the computational complexity of the decision problem related to computing the (total, dual and outer) mutual-visibility chromatic number, i.e., the next problems.

\begin{center}
\begin{boxedminipage}{0.85\textwidth}
{\sc (TOTAL, DUAL, OUTER) MV-coloring} problem: \\
{\sc Instance} $(G,k)$: A graph $G$, and a positive integer $k\leq |V(G)|$. \\
{\sc Question}: Does $G$ contain a (total, dual, outer) mutual-visibility $k$-coloring?
\end{boxedminipage}
\end{center}

\bigskip
We prove the NP-completeness of the problem above, by using a polynomial-time reduction from the {\em not-all-equal 3-satisfiability} ({\sc NAE3-SAT}) problem, which is known to be NP-complete (\cite{schafer-1978}).

A {\sc NAE3-SAT} (\emph{not-all-equal 3-satisfiability}) instance $\Phi=(X,C)$ is defined over a set $X=\{x_1,\ldots,x_q\}$ of $q$ Boolean variables and a collection $C$ of $r$ clauses, each  defined as a set of three literals: every variable $x_i$ corresponds to two literals, that is, $x_i$ (the positive form) and $\bar x_i$ (the negative form). To simplify the notations we will denote by $\{\ell_1, \ell_2, \ell_3\}$ the clause with literals $\ell_i, i\in[3]$, without distinction between the orders in which they are listed. A truth assignment assigns a Boolean value ($True$ or $False$) to each variable, corresponding to a truth assignment of opposite values for the two literals $x_i$ and $\bar {x_i}$: $\bar{x_i}$ is $True$ if and only if $x_i$ is $False$. 
	
The {\sc NAE3-SAT} problem asks whether there is a truth assignment to the variables such that in no clause all three literals have the same truth value. We will say that such an assignment is \emph{satisfying} and the instance $\Phi $ is \emph{satisfied}.

\begin{theorem}\label{thm:mvc-NP}
The {\sc Total MV-Coloring} problem is NP-complete even for instances $(G,k)$ with $k=2$.
\end{theorem}

\begin{proof}
The {\sc Total MV-Coloring} problem is in NP since, given a coloring of $G$, it is possible to check in polynomial time  whether each color class is a total mutual-visibility set.

\medskip
In order to proceed with our arguments, given an instance $\Phi=(X,C)$ on $q$ variables and $r$ clauses, we construct a graph $G=(V,E)$ and prove that such $G$ admits a total mutual-visibility $2$-coloring if and only if $\Phi$ is satisfiable. We proceed as follows.
\begin{itemize}
    \item For each variable $x_i$ ($i\in[q]$), introduce a gadget isomorphic to $K_4$ on vertices $\{u_i, {u}_i', v_i, {v}_i'\}$.
    \item For each clause $c_j$ ($j\in[r]$), introduce two vertices $c_j$ and $c_j'$, where $c_j'$ is adjacent to exactly the same vertices as $c_j$. Connect $c_j$ (and thus $c_j'$) to the appropriate literal vertex: to $u_{i}$ if $x_{i}\in c_j$, or to ${u}_{i}'$ if $\bar{x}_{i}\in c_j$. Let $S_j$ be the set of these three vertices (one per literal of $c_j$).
    \item Each $v_i$ and ${v}_i'$ are each connected to every $c_j$ ($j\in[r]$), and so, to each $c_j'$.
    \item Additionally, introduce vertices $x$ and $y$, and add edges to both $u_i$ and ${u}_i'$ for every $i\in[q]$.
\end{itemize}

The construction is clearly polynomial in the size of $\Phi$. An example of the construction is drawn in Figure \ref{fig:complex-Tot}. 

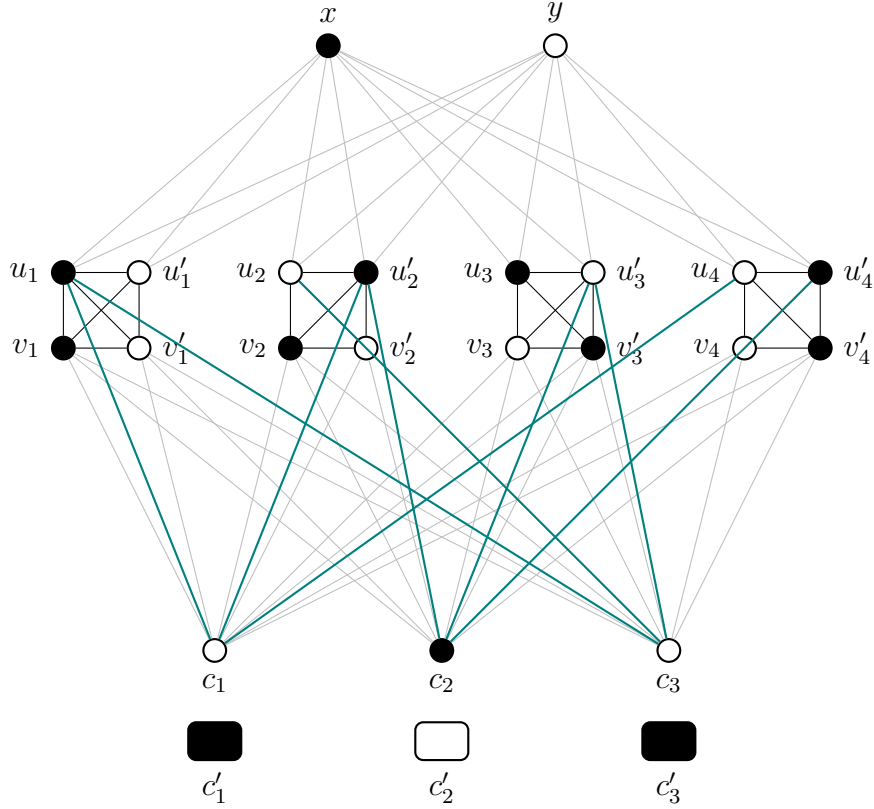
\begin{figure}[ht] 
\centering
\begin{tikzpicture}[
     vertex/.style={circle, draw,thick,minimum size=3mm,inner sep=0pt,font=\scriptsize},
     copy/.style={rectangle, draw, thick, rounded corners=3pt, minimum width=7mm, minimum height=5mm, font=\scriptsize},
     scale=1.0,
        every label/.style={font=\normalsize}
    ]

     \node[vertex ,fill=black, label=left:$v_1$] (v1)  at (-4,0)  {};
     \node[vertex , fill=black,label=left:$u_1$] (v2)  at (-4,1)  {};
     \node[vertex ,label=right:$v_1'$] (v3)  at (-3,0)  {};
     \node[vertex ,label=right:$u_1'$] (v4)  at (-3,1)  {};

  \node[vertex ,fill=black,label=left:$v_2$] (v5)  at (-1,0)  {};
     \node[vertex ,label=left:$u_2$] (v6)  at (-1,1)  {};
     \node[vertex ,label=right:$v_2'$] (v7)  at (0,0)  {};
     \node[vertex ,fill=black, label=right:$u_2'$] (v8)  at (0,1)  {};

     \node[vertex ,label=left:$v_3$] (v9)  at (2,0)  {};
     \node[vertex ,fill=black,label=left:$u_3$] (v10)  at (2,1)  {};
     \node[vertex ,fill=black,label=right:$v_3'$] (v11)  at (3,0)  {};
     \node[vertex ,label=right:$u_3'$] (v12)  at (3,1)  {};

     \node[vertex ,label=left:$v_4$] (v13)  at (5,0)  {};
     \node[vertex ,label=left:$u_4$] (v14)  at (5,1)  {};
     \node[vertex ,fill=black,label=right:$v_4'$] (v15)  at (6,0)  {};
     \node[vertex ,fill=black,label=right:$u_4'$] (v16)  at (6,1)  {};

     \node[vertex , fill=black,label=above:$x$] (v17)  at (-0.5,4)  {};
     \node[vertex ,label=above:$y$] (v18)  at (2.5,4)  {};

    \node[vertex ,label=below:$c_1$] (v19)  at (-2,-4)  {};
     \node[vertex ,fill=black,label=below:$c_2$] (v20)  at (1,-4)  {};
     \node[vertex ,label=below:$c_3$] (v21)  at (4,-4)  {};

  \node[copy,fill=black, label=below:$c_1'$] (v22) at (-2,-5.2) {};
     \node[copy, label=below:$c_2'$] (v23) at (1,-5.2) {};
     \node[copy, fill=black, label=below:$c_3'$] (v24) at (4,-5.2) {};

     \foreach \p in {v1,v3,v5,v7,v9,v11,v13,v15} {
        \draw [gray!50, thin](v19) -- (\p);
        \draw [gray!50, thin](v20)  -- (\p);
        \draw [gray!50, thin](v21)  -- (\p);
    }

   \foreach \p in {v2,v4,v6,v8,v10,v12,v14,v16} {
        \draw [gray!50, thin](v17) -- (\p);
        \draw[gray!50, thin] (v18)  -- (\p);
}
    
 \draw (v1)-- (v2) -- (v3) -- (v4) -- (v2);
  \draw (v5)-- (v6) -- (v7) -- (v8) -- (v6);
   \draw (v9)-- (v10) -- (v11) -- (v12) -- (v10);
    \draw (v13)-- (v14) -- (v15) -- (v16) -- (v14);
 \draw (v3) -- (v1) -- (v4);
  \draw (v7) -- (v5)-- (v8);
   \draw (v11) -- (v9)-- (v12);
    \draw (v15) -- (v13)-- (v16);

\draw [teal,thick](v19)-- (v2);
\draw [teal,thick](v19) -- (v8);
\draw [teal,thick](v19) -- (v14);

\draw [teal,thick](v20)-- (v8);
\draw [teal,thick](v20) -- (v12);
\draw [teal,thick](v20) -- (v16);

\draw [teal,thick](v21)-- (v2);
\draw [teal,thick](v21) -- (v6);
\draw [teal,thick](v21) -- (v12);

\end{tikzpicture}

\caption{\small An illustration of the graph $G$, constructed in the proof of Theorem \ref{thm:mvc-NP}. Here, $X=\{x_1,x_2,x_3,x_4\}$, $c_1=\{x_1,\bar{x_2},x_4\},c_2=\{\bar{x_2},\bar{x_3},\bar{x_4}\},c_3=\{x_1,x_2,\bar{x_3}\}$. Each $c_j'$ is represented by a box and is adjacent to all the neighbors of $c_j$. For clarity of the picture these adjacencies are not drawn.  }\label{fig:complex-Tot}
\end{figure}

Observe first that there exist exactly three shortest paths from $x$ (resp.\ $y$) to each $c_j$ whose internal vertices are precisely the vertices in $S_j$. Therefore, in any total mutual-visibility coloring, the vertices of $S_j$ cannot all receive the same color (otherwise a geodesic from $x$ or $y$ to any $c_j$ would be internally blocked by that color class).

We now claim that $\Phi$ is satisfiable if and only if $G$ admits a total mutual-visibility $2$-coloring.

\medskip
($\Rightarrow$) Suppose $t:X\to\{\text{True},\text{False}\}$ is a satisfying truth assignment for $\Phi$. Color the vertices of $G$ with two colors (black and white) as follows:
\begin{itemize}
    \item For each $i$, color $u_i$ white and ${u}_i'$ black if $t(x_i)=\text{True}$ (and vice versa if $t(x_i)=\text{False}$).
    \item Color pairs such as $\{v_i,{v}_i'\}$ and $\{c_j,c_j'\}$ with opposite colors.
    \item Color $x$ and $y$ with opposite colors (one black and one white)
\end{itemize}

The NAE (not-all-equal) property of $t$ ensures that no set $S_j$ is monochromatic. Due to the mixed colors in the variable and clause gadgets together with the alternative routing provided by the construction, for every pair of vertices in $G$ there exists a shortest path whose internal vertices avoid each color class. Thus, both color classes are total mutual-visibility sets.

\medskip
($\Leftarrow$) Conversely, assume that there exists a total mutual-visibility $2$-coloring of $G$ with colors black and white. Hence, no matter which colors the vertices of a pair $z,z'$ have, where (w.l.o.g.) $z\in \{x,y\}$ and $z'=v_j$ for some $j\in [q]$, such two vertices $z,z'$ must be visible with respect to each color class (black and white). Thus, within each $K_4$ gadget corresponding to a variable $i\in [q]$, the vertices $u_i$ and $u_i'$ must receive different colors (enforced by the structure and connections between $z$ and $z'$). Define a truth assignment $t$ by setting $t(x_i)=\text{True}$ if $u_i$ is white and $t(x_i)=\text{False}$ if $u_i$ is black.  

Suppose for contradiction that $\Phi$ is not satisfiable. Then there exists a clause $c_h$ whose three literals all have the same truth value under $t$. This implies that the three vertices in $S_h$ are all the same color, say white. But then all shortest paths from $z$ to $c_h$ have internal vertices belonging to the white class, contradicting the assumption that the white class is a total mutual-visibility set.

Therefore, $t$ is a satisfying NAE assignment for $\Phi$. This completes the reduction.
\end{proof}

\begin{theorem}\label{thm:dmvc-NP}
The {\sc Dual MV-Coloring} problem is NP-complete even for instances $(G,k)$ with $k=2$.
\end{theorem}

\begin{proof}
Given an instance $\Phi=(X,C)$ on $q$ variables and $r$ clauses, consider the same graph constructed in Theorem \ref{thm:mvc-NP}. The ($\Rightarrow$) part follows similarly to that of Theorem \ref{thm:mvc-NP}, since every total mutual-visibility set is a dual mutual-visibility set.\par

($\Leftarrow$) Conversely, assume that there exists a dual mutual-visibility $2$-coloring of $G$ with colors black and white. If $x$ and $y$ are of the same color, say black, then there exists at least one $u_i$ or $u_i'$ with color white. Without loss of generality, let $u_1$ be colored white. Thus, all the vertices $u_i,u_i'$, with $i \in [q] \setminus \{1\}$, are not visible from $u_1$ and hence must be black. But, then $u_j$ and $u_k$ are not visible, for $j \neq k \in [q] \setminus \{1\}$, which is a contradiction. Therefore, $x$ and $y$ must be of different colors. Also, since $v_i$ must be visible either to $x$ or to $y$, it holds that $u_i$ and $u_i'$ must be of opposite colors for each $i \in [q]$. Define a truth assignment $t$ by setting $t(x_i)=\text{True}$ if $u_i$ is white and $t(x_i)=\text{False}$ if $u_i$ is black, and the rest of the proof follows similarly to that of Theorem \ref{thm:mvc-NP}.
\end{proof}

\begin{theorem}\label{thm:omvc-NP}
The {\sc Outer MV-Coloring} problem is NP-complete even for instances $(G,k)$ with $k=2$.
\end{theorem}

\begin{proof}
Given an instance $\Phi=(X,C)$ on $q$ variables and $r$ clauses, consider the same graph constructed in Theorem \ref{thm:mvc-NP}. The ($\Rightarrow$) part follows similarly to that of Theorem \ref{thm:mvc-NP}, since every total mutual-visibility set is an outer mutual-visibility set.\par

($\Leftarrow$) Conversely, assume that there exists an outer mutual-visibility $2$-coloring of $G$ with colors black and white. Whatever be the color of $x$, all $v_i$'s must be visible with $x$ and hence, $u_i$ and $u_i'$ must have opposite colors for each $i \in [q]$. The rest of the proof follows similarly to that of Theorem \ref{thm:mvc-NP}.
\end{proof}

\section{Block graphs}\label{sec:block}

Recall that a block in a graph $G$ is a maximal biconnected subgraph. In particular, along our work, a graph is a \textit{block graph} if all its blocks are cliques (a complete subgraph). A tree is a special type of block graph where each block is an edge. Observe that  every block graph $G$ with at least two blocks, contains one vertex which is the center of a convex $P_3$. Thus by Theorem \ref{th_chimut-inf}, $\chimut(G)=\infty$ for every block graph with at least two blocks.  We next consider the other two parameters $\chimud$ and $\chimuo$ for block graphs. To see this, we recall that a vertex is called \textit{simplicial} if the set of its neighbors form a clique. Also, if $v\in V(G)$, by $G-v$ we represent the graph obtained from $G$ after removing the vertex $v$ and the edges incident with $v$. Clearly, if $v$ is a cut vertex, then $G-v$ has a larger number of components than $G$. If $v\in V(G)$ is a cut vertex, then by $cd(v)$ we represent the number of components of $G-v$.

\begin{proposition}
\label{prop:block}
If  $G$ is a connected graph and $\Omega$ is the set of its cut vertices, then 
$$\chimuo(G)\geq|\Omega|+1$$ 
Moreover, if $G$ is a block graph, then the equality is achieved.
\end{proposition}

\begin{proof}
Let $c:V(G)\rightarrow [k]$ be an outer mutual-visibility $k$-coloring. We claim that if $c(v)=i$ and $v\in \Omega$, then $c(u) \neq i$ for every $u \in V(G)\setminus \{v\}$. Suppose, for contradiction, that  $c(u)=i$ for some $u \in V(G)\setminus \{v\}$. Since $v$ is a cut vertex, $G-v$ has at least two components. Let $C$ be a component of $G-v$ not containing $u$ and let $w \in C$ . Then any shortest path from $u$ to $w$ contains $v$ as an internal vertex, which violates the outer mutual-visibility property for the color class containing $u$ and $v$. Thus, no other vertex can share the color of a cut vertex. Namely, each cut vertex forms a singleton class in any outer mutual-visibility $k$-coloring.
 
Additionally, every connected graph with at least two vertices has at least two non-cut vertices. Therefore, at least $|\Omega|+1$ colors are required for an outer mutual-visibility coloring of $G$. Hence, $\chimuo(G)\geq|\Omega|+1$ .
 
Now, assume that $G$ is a block graph. Let $S$ be the set of all simplicial vertices of $G$. It is clear that $S$ and singletons $\{v\}$, for each $v \in \Omega$, are pairwise disjoint outer mutual-visibility sets whose union is $V(G)$. Thus, $\chimuo(G) \leq |\Omega|+1$, and the desired equality holds.   
\end{proof} 

Trees are particular cases of block graphs. In such situation, only leaves are not cut vertices. By Theorem \ref{th_chimut-inf}, $\chimut(T)=\infty$ for every tree $T$ of order at least $3$. Also, Proposition \ref{prop:block} above reads as follows for trees. 

\begin{corollary}
For any tree $T$ of order $n$ and $\ell$ leaves, $\chimuo(T)=n-\ell+1$.
\end{corollary}

We next study the dual mutual-visibility chromatic number of block graphs, in which the set of cut vertices $\Omega$ plays an important role.
%A block graph $G$ belongs to the family $\mathcal{F}$, if $G$ is obtained from two cliques $K_r$ and $K_s$ by selecting one vertex from each clique and either (i) identifying such two vertices, or (ii) adding an edge between these two vertices. Notice that any graph from $\mathcal{F}$ has $1$ or $2$ cut vertices, and for instance, $P_3$ and $P_4$ belong to $\mathcal{F}$.

\begin{proposition}
\label{prop:block-dual}
If $G$ is a block graph with the set of cut vertices $\Omega$, then
$$\chimud(G)=\left\{\begin{array}{ll}
    1; & \mbox{if $G$ is a complete graph}, \\[0.1cm]
    %2; & \mbox{if $G\in \mathcal{F}$}, \\[0.1cm]
    |\Omega|; & \mbox{if $\Omega$ forms a block of $G$ and $cd(x)=2$ for every $x\in \Omega$},\\[0.1cm]
    |\Omega|+1; & \mbox{if $\Omega$ induces a proper subgraph of a block of $G$}\\
    & \mbox{and $cd(x)=2$ for every $x\in \Omega$},\\[0.1cm]
    \infty; & \mbox{otherwise}.
\end{array}\right.$$ 
\end{proposition}

\begin{proof}
If $G$ is a complete graph, then it clearly happens that $\chimud(G)=1$. If $G$ has only one cut vertex $x$ with $cd(x)=2$, then $G$ is a block graph obtained from two cliques, say $K_r$ and $K_s$, by selecting one vertex from each clique and identifying such two vertices (once identified in $G$, this vertex is $x$). The set of vertices of these cliques are then the blocks of $G$. One can readily see that $\{V(K_r),V(K_s)\setminus \{x\}\}$ forms a dual mutual-visibility $2$-coloring of $G$, and so, $\chimud(G)\le 2$, which is indeed an equality since $G$ is not complete. On the other hand, if $G$ has only one cut vertex $x$ with $cd(x)>2$, then $G$ contains a convex $K_{1,3}$, and by Theorem \ref{th_chimud-inf}, we deduce that $\chimud(G)=\infty$. In this sense, from now on, we assume $|\Omega|\ge 2$.

If $G$ has at least $2$ cut vertices at distance at least two, or it has a cut vertex $v$ such that $cd(v)\ge 3$, then $G$ contains a convex $P_5$ or a convex $K_{1,3}$, respectively. Thus, by Theorem \ref{th_chimud-inf}, we deduce that $\chimud(G)=\infty$. Hence, we may now consider that all the cut vertices of $G$ are pairwise adjacent (i.e., they induce a complete graph), and also that $cd(x)=2$ for every $x\in \Omega$. We indeed have that $G$ is formed by a block with vertex set $B$ containing all the cut vertices, and for any two cut vertices $x,y\in B$, there are two blocks (different from $K_1$) with vertex sets $B_x\ne B$ and $B_y\ne B$ such that $x\in B_x$, $x,y\in B$ and $y\in B_y$. Let $\Omega=\{x_1,\dots,x_r\}$ with $r=|\Omega|\ge 2$, and for any vertex $x_i\in \Omega$, let $B_{i}$ be the vertex set of the block containing $x_i$ different from $B$.

Consider any cut vertex $x_i\in \Omega$ and let $S_\ell$ be a dual mutual-visibility set of $G$ such that $x_i\in S_\ell$. If $S_\ell\cap (B_i\setminus \{x_i\})\ne\emptyset$, then it holds $S_\ell\subseteq B_i$ (for otherwise a vertex not in $B_i$ and other vertex outside $B_i$ are not $S_\ell$-visible). In addition, if there is a dual mutual-visibility set $S_\ell$ such that it contains at least two cut vertices $x_i,x_j$, then $S_\ell\cap B_i=\emptyset$ and $S_\ell\cap B_j=\emptyset$, but then any two vertices $u\in B_i$ and $v\in B_j$ (which are not in $B_\ell$) are not $S_\ell$-visible, which is not possible. Based on this fact, we deduce that each cut vertex of $G$ must belong to a different dual mutual-visibility color class in any dual mutual-visibility coloring of $G$, and so, $\chimud(G)\ge r$.  

If the block with vertex set $B$ is isomorphic to $K_{r}$ (all the vertices in $B$ are cut vertices), then it can be readily observed that the collection $\left\{B_1,B_2,\dots, B_r\right\}$ forms a dual mutual-visibility $r$-coloring of $G$. Thus, $\chimud(G)\le r$, and the desired equality $\chimud(G) = r = |\Omega|$ follows in this case. 

If the block with vertex set $B$ is not isomorphic to $K_{r}$, then there is a vertex $w\in B \setminus\Omega$ which is adjacent to every vertex in $\Omega$ (clearly such $w$ is not a cut vertex). In such a situation, it can be readily observed that the collection $\left\{A,B_1,B_2,\dots, B_r\right\}$ with $A=B\setminus \left(\bigcup_{i\in [r]}B_i\right)$ forms a dual mutual-visibility $(r+1)$-coloring of $G$. Thus, $\chimud(G)\le r+1$. 

Suppose $\chimud(G)=r$ and let $\{S_1,\dots,S_r\}$ be the respective dual mutual-visibility color classes of $G$. Since each two cut vertices cannot belong the a same class $S_\ell$, we can w.l.o.g. assume that $x_i\in S_i$ for every $i\in [r]$. Hence, there is a class $S_q$ such that $w\in S_q$ and clearly $S_q$ contains the cut vertex $x_q$. However, in such situation, for every $z\in B_q$, it follows that $w$ and $z$ are not $S_q$-visible, which is a contradiction. Thus, $\chimud(G)=r$ is not possible, and we deduce that $\chimud(G)\ge r+1$, which leads to the desired equality $\chimud(G) = r+1= |\Omega|+1$ in this setting. This argument completes the whole proof.
\end{proof}

The particular case of trees in the result above can be then written as follows.

\begin{corollary}
For any tree $T$ of order $n\ge 3$, 
$$\chimud(T)=\left\{\begin{array}{ll}
    2; & \mbox{if $T=P_3$ or $T=P_4$}, \\
    \infty; & \mbox{otherwise}.
\end{array}\right.$$   
\end{corollary}

\section{Hamming graphs}\label{sec:Hamming}

% \begin{proposition}
%     Let $G$ be a connected graph with at least two vertices and $H$ be any graph with at least two vertices. Then $\chimut(G \circ H)=2$ whenever at least one of $G$ or $H$ is not a complete graph.
% \end{proposition}

The $2$-dimensional Hamming graphs (from now on Hamming graphs for short) are obtained as the Cartesian product $K_m\cp K_n$, where $m,n\ge 2$. In this sense, we assume next that $V(K_m \cp K_n)=\{(i,j): i\in [m], j\in[n]\}$. Before turning to our results, we briefly discuss what is known about their mutual-visibility chromatic number. Its exact determination appears to be difficult, as it is closely related to the well-known Zarankiewicz problem. In fact, it was proved in \cite{CDK} that $\mu(K_m\cp K_n)=z(m,n;2,2)$, where $z(m,n;2,2)$ denotes the maximum number of $1$'s in an $m \times n$ binary matrix containing no $2 \times 2$ submatrix of $1$'s. The determination of $z(m,n;2,2)$ is a special case of the Zarankiewicz problem, and its exact value is not known in general. Consequently, by inequality \eqref{eq:chimus-lower-bounds},
$$\chimu(K_m\cp K_n) \ge \left\lceil\frac{mn}{z(m,n;2,2)}\right\rceil.$$
In the case $m=n$, known bounds on $z(n,n;2,2)$ give a more explicit lower bound. For sufficiently large $n$, it is known that $z(n,n;2,2)\le \frac{n}{2}\left(1+\sqrt{4n-3}\right)$ and hence
$$\chimu(K_n\cp K_n)\ge\left\lceil\frac{2n}{1+\sqrt{4n-3}}\right\rceil.$$
In fact, it was proved in \cite{KKVY} that
$\chimu(K_n\cp K_n)=\Theta(\sqrt n)$ for sufficiently large $n$.

While the exact value of $\chimu(K_m\cp K_n)$ therefore seems difficult to determine in general, the situation is more tractable for its variants. We first determine the total and dual mutual-visibility chromatic numbers of $K_m\cp K_n$, and lastly, we consider the outer mutual-visibility chromatic number of $K_m\cp K_n$, for which the situation is more involved and leads to a connection with the star arboricity of complete bipartite graphs.

\begin{theorem}\label{th:hamm_chimut}
Let $m,n\ge 2$. Then $\chimut(K_m \cp K_n)=\min\{m,n\}$.
\end{theorem}

\begin{proof}
Without loss of generality, assume that $m\ge n$. It is known from \cite{TK} that
$$\mut(K_m \cp K_n)=\max\{m,n\}=m.$$
Since $n(K_m \cp K_n)=mn$, the lower bound in \eqref{eq:chimus-lower-bounds} gives
$$\chimut(K_m \cp K_n) \ge \left\lceil\frac{mn}{m}\right\rceil=n.$$

For the reverse inequality, for every $j\in [n]$, let $M_j=\{(i,j): i\in [m]\}$. The sets $M_1,\ldots,M_n$ form a partition of $V(K_m \cp K_n)$. We claim that each $M_j$ is a total mutual-visibility set. Fix $j\in [n]$ and let $x=(a,b)$ and $y=(c,d)$ be two vertices of $K_m \cp K_n$. If $x$ and $y$ are adjacent, then the edge $xy$ is a geodesic with no internal vertices, and hence $x$ and $y$ are $M_j$-visible. Suppose now that $x$ and $y$ are not adjacent. Then $a\neq c$ and $b\neq d$, and therefore $d(x,y)=2$. There are two geodesics,
$$(a,b),(a,d),(c,d) \qquad\mbox{and}\qquad (a,b),(c,b),(c,d),$$
between $x$ and $y$. Their internal vertices are $(a,d)$ and $(c,b)$, respectively. Since $b\neq d$, at most one of these two vertices belongs to $M_j$. Hence, at least one of the two geodesics has no internal vertex in $M_j$, and thus $x$ and $y$ are $M_j$-visible.

Therefore, each $M_j$ is a total mutual-visibility set, and $\{M_1,\ldots,M_n\}$ defines a total mutual-visibility $n$-coloring of $K_m \cp K_n$. Consequently, $\chimut(K_m \cp K_n)\le n$. Together with the lower bound, this gives $\chimut(K_m \cp K_n)=n=\min\{m,n\}$.
\end{proof}

\medskip

The lower bound in \eqref{eq:chimus-lower-bounds} also gives some information about the dual mutual-visibility chromatic number of the Cartesian product of two complete graphs. Assume that $m\geq n$. Since $\mud(K_m\cp K_n)=m+n-1$ by \cite{CDKY2}, we obtain
$$\chimud(K_m \cp K_n) \ge \left\lceil\frac{mn}{m+n-1}\right\rceil.$$
If $n$ is fixed and $m$ tends to infinity, then
$$\left\lceil\frac{mn}{m+n-1}\right\rceil \xrightarrow{m \to \infty} n$$
On the other hand, when the two factors have almost the same order, the bound is considerably weaker. In particular, for $m=n$ it gives
$$\chimud(K_n \cp K_n) \ge \left\lceil\frac{n^2}{2n-1}\right\rceil=
\begin{cases}
\left\lceil\frac{n}{2}\right\rceil+1; & \text{$n$ is even},\\
\left\lceil\frac{n}{2}\right\rceil; & \text{$n$ is odd}.
\end{cases}$$

The next result shows that the exact value is nevertheless $n$ for all $m\geq n$. Thus, the general lower bound is asymptotically tight when one factor is much larger than the other, while for factors of (almost) equal order it may be far from the exact value.

\begin{theorem}\label{th:hamm_chimud}
Let $m,n\ge 2$. Then $\chimud(K_m \cp K_n)=\min\{m,n\}$.
\end{theorem}

\begin{proof}
Without loss of generality, assume that $m \ge  n$. For every $i \in [m]$ and $j \in [n]$, let
$$R_i=\{(i,\ell) : \ell \in [n]\} \qquad \mbox {and} \qquad C_j=\{(\ell,j) : \ell \in [m]\}$$
denote the corresponding row and column, respectively. 
The upper bound $\chimud(K_m \cp K_n) \le n$ follows directly from Theorem \ref{th:hamm_chimut}, since any total mutual-visibility coloring is also a dual mutual-visibility coloring.

%For every $j \in [n]$, let $M_j=C_j$. Any two vertices of $M_j$ are adjacent, hence, they are $M_j$-visible. Now let $x=(a,b)$ and $y=(c,d)$ be two vertices outside $M_j$. If $x$ and $y$ are adjacent, then they are again $M_j$-visible. Therefore, suppose that they are not adjacent. Then $a \neq c$ and $b \neq d$, and the path $(a,b),(a,d),(c,d)$ is an $(x,y)$-geodesic. Its unique internal vertex $(a,d)$ does not belong to $M_j$, since $y \notin M_j$ implies that $d \neq j$. Thus, $x$ and $y$ are $M_j$-visible. It follows that $M_j$ is a dual mutual-visibility set. Since $M_1, \ldots, M_n$ form a partition of $V(K_m \cp K_n)$, assigning one color to each of these sets gives a dual mutual-visibility $n$-coloring. Therefore,
%$$\chimud(K_m \cp K_n) \le n.$$

We next prove the lower bound. Suppose, to the contrary, that $K_m \cp K_n$ admits a dual mutual-visibility $k$-coloring, where $k < n$. Let $c: V(K_m \cp K_n) \to [k]$ be such a coloring, and let $M_\alpha=c^{-1}(\alpha)$ be the color class corresponding to color $\alpha \in[k]$. Choose a set $J \subseteq [n]$ of $k+1$ columns, which is possible since $k < n$. For every row $R_i$, the $k+1$ vertices of the set $\{(i,j) : j \in J\}$ are colored with only $k$ colors. Hence, by the pigeonhole principle, there exists a color $\alpha_i \in [k]$ which appears at least twice in row $R_i$ among the columns in $J$. We claim that
$\alpha_i \neq \alpha_{i'}$ for any two distinct rows $R_i$ and $R_{i'}$ (i.e., there do not exist two distinct rows with the same repeated color). Suppose otherwise. Then there exist distinct $i,i' \in [m]$ and a color $\alpha \in [k]$ such that $\alpha_i=\alpha_{i'}=\alpha$. Let
$$A=\{j \in J :  c(i,j)=\alpha\} \qquad \mbox{and}\qquad B=\{j\in J : c(i',j)=\alpha\}.$$
By the choice of $\alpha$, we have $|A| \ge 2$ and $|B| \ge 2$.

First suppose that
$|A\cap B|\geq2$. Choose distinct $p,q \in A \cap B$. Then all four vertices $(i,p)$, $(i,q)$, $(i',p)$, $(i',q)$ belong to $M_\alpha$. Consider the diagonal vertices $x=(i,p)$ and $y=(i',q)$. Since $i \ne i'$ and $p \ne q$, we have $d(x,y)=2$. Moreover, the only two $(x,y)$-geodesics are
$$(i,p),(i,q),(i',q) \qquad \mbox{and} \qquad (i,p),(i',p),(i',q).$$
Their internal vertices $(i,q)$ and $(i',p)$ both belong to $M_\alpha$. Hence, every $(x,y)$-geodesic contains an internal vertex from $M_\alpha$, and therefore $x$ and $y$ are not $M_\alpha$-visible. This contradicts the fact that $M_\alpha$ is a dual mutual-visibility set.

It remains to consider the case $|A \cap B| \le 1$. Since $|A| \ge 2$ and $|B| \ge 2$, there exist $p \in A \setminus B$ and $q \in B \setminus A$ (note that $p \neq q$). Thus, $(i,p),(i',q) \in M_\alpha$, while $(i,q),(i',p) \notin M_\alpha$. Consider the two vertices $x=(i,q)$ and $y=(i',p)$. Again, $d(x,y)=2$, and the only two $(x,y)$-geodesics are
$$(i,q),(i,p),(i',p) \qquad \mbox{and} \qquad (i,q),(i',q),(i',p).$$
The internal vertices $(i,p)$ and $(i',q)$ both belong to $M_\alpha$. Consequently, every $(x,y)$-geodesic contains an internal vertex from $M_\alpha$, so $x$ and $y$ are not $M_\alpha$-visible. Since both vertices lie outside $M_\alpha$, this again contradicts the fact that $M_\alpha$ is a dual mutual-visibility set.

We have proved that $\alpha_i \neq \alpha_{i'}$ whenever $i \neq i'$. Recall that, by the pigeonhole principle, every row $R_i$ contains a color $\alpha_i$ which appears at least twice among the $k+1$ selected columns. Hence, each of the $m$ rows has at least one repeated color. By what we have just proved, the same color cannot be chosen as a repeated color in two distinct rows. Therefore, the $m$ rows require $m$ distinct repeated colors. However, the coloring uses only $k$ colors in total, therefore $m \le k$. On the other hand, by our assumptions, it holds $k < n \le m$, which gives a contradiction. Thus, no dual mutual-visibility coloring of $K_m \cp K_n$ with fewer than $n$ colors exists, and hence $\chimud(K_m \cp K_n) \ge n$. Together with the upper bound, this yields $\chimud(K_m\cp K_n)=n=\min\{m,n\}$.
\end{proof}

\medskip

We now turn our attention to the outer mutual-visibility chromatic number of $K_m \cp K_n$. As in the dual case, it is useful to first consider the general lower bound in \eqref{eq:chimus-lower-bounds}. Assume that $m\ge n$. Since $\muo(K_m\cp K_n)=m+n-2$ by \cite{CDKY2}, we obtain
\begin{equation}\label{eq:chimuo-lower-bound}
\chimuo(K_m \cp K_n) \ge \left\lceil\frac{mn}{m+n-2}\right\rceil.
\end{equation}
For fixed $n$, as $m$ tends to infinity, we have
$$\left\lceil\frac{mn}{m+n-2}\right\rceil \xrightarrow{m\to\infty} n,$$
whereas for $m=n$,
\begin{equation}\label{eq:chimuo-m=n-lower-bound}
\chimuo(K_n \cp K_n) \ge \left\lceil\frac{n^2}{2n-2}\right\rceil=\left\lceil\frac{n}{2}\right\rceil+1.
\end{equation}
For the dual mutual-visibility chromatic number, Theorem~\ref{th:hamm_chimud} shows that the exact value is always $\min\{m,n\}$. However, the situation for the outer mutual-visibility chromatic number is more involved. It turns out that the problem can be naturally related to another well-known graph invariant called
the star arboricity.

A \emph{star forest} is a forest whose connected components are stars.
The \emph{star arboricity} of a graph $G$, denoted by $\sa(G)$, is the
minimum number of star forests whose union covers all edges of $G$. This notion was first introduced in \cite{AK}. The following theorem establishes a direct connection between
the outer mutual-visibility chromatic number of the Cartesian product of two complete graphs and the star arboricity of a complete bipartite graph.

\begin{theorem}\label{th:hamm_chimuo}
Let $m,n\ge 2$. Then $\chimuo(K_m\cp K_n)=\sa(K_{m,n})$.
\end{theorem}

\begin{proof}
Recall first that $K_m \cp K_n$ is isomorphic to the line graph ${\mathcal{L}}(K_{m,n})$. Let the partite sets of $K_{m,n}$ be $X=\{x_1,\ldots,x_m\}$ and $Y=\{y_1,\ldots,y_n\}$. Then the isomorphism $\varphi: V(K_m \cp K_n) \to E(K_{m,n})$ is given by
$$\forall i \in [m], \forall j \in [n]: \varphi((i,j))=x_i y_j.$$
Note that two vertices $(i,j)$ and $(i',j')$ are adjacent in $K_m \cp K_n$ if and only if $i=i'$ or $j=j'$. This is equivalent to the corresponding edges $x_i y_j$ and $x_{i'}y_{j'}$ having a common endvertex in $K_{m,n}$, which is exactly the adjacency relation in the line graph.

For a subset $M \subseteq V(K_m\cp K_n)$, let $H_M$ be the spanning subgraph of $K_{m,n}$ with
$$E(H_M)=\{\varphi((i,j)): (i,j) \in M\}.$$
We first show that $M$ is an outer mutual-visibility set of  $K_m \cp K_n$ if and only if $H_M$ is a star forest.

Suppose first that $M$ is an outer mutual-visibility set. We claim that $H_M$ does not contain a $P_4$ as a subgraph. Suppose otherwise. Since $H_M$ is bipartite, such a path can be written as $x_iy_px_{i'}y_q$ for some distinct $i,i' \in [m]$ and distinct $p,q \in [n]$. Hence, $(i,p),(i',p),(i',q) \in M$. Suppose first that $(i,q) \in M$, and consider the diagonal vertices $u=(i,p)$ and $v=(i',q)$. Since $i \neq i'$ and $p\neq q$, we have $d(u,v)=2$. The only two $(u,v)$-geodesics are
$$(i,p),(i,q),(i',q) \qquad\mbox{and} \qquad (i,p),(i',p),(i',q).$$
Their internal vertices $(i,q)$ and $(i',p)$ both belong to $M$. Therefore, every $(u,v)$-geodesic contains an internal vertex from $M$, and thus $u$ and $v$ are not $M$-visible. Since both $u$ and $v$ belong to $M$, this contradicts the fact that $M$ is an outer mutual-visibility set. Suppose now that $(i,q) \notin M$, and consider the other diagonal pair $u=(i,q)$ and $v=(i',p)$. For them $u \notin M$ and $v \in M$. Again, $d(u,v)=2$, and the only two $(u,v)$-geodesics are
$$(i,q),(i,p),(i',p) \qquad\mbox{and}\qquad (i,q),(i',q),(i',p).$$
Their internal vertices $(i,p)$ and $(i',q)$ both belong to $M$. Hence, $u$ and $v$ are not $M$-visible. Since one of the two vertices belongs to $M$, this again contradicts the outer mutual-visibility property. Therefore, $H_M$ contains no $P_4$ as a subgraph. Every connected bipartite graph which is not a star contains a $P_4$ as a subgraph. Hence, every nontrivial connected component of $H_M$ is a star, and therefore $H_M$ is a star forest.

For the converse, suppose that $H_M$ is a star forest. We prove that $M$ is an outer mutual-visibility set. Let $u=(i,p)$ and $v=(i',q)$ be two vertices of $K_m \cp K_n$ such that at least one of them belongs to $M$. If $u$ and $v$ are adjacent, then the edge $uv$ is a geodesic with no internal vertices, and hence $u$ and $v$ are $M$-visible. Therefore, suppose that $u$ and $v$ are not adjacent. Then $i \neq i'$ and $p \neq q$, so $d(u,v)=2$. The only two $(u,v)$-geodesics are
$$(i,p),(i,q),(i',q) \qquad\mbox{and}\qquad (i,p),(i',p),(i',q).$$
Suppose, to the contrary, that $u$ and $v$ are not $M$-visible. Then the internal vertices of both geodesics belong to $M$, that is $(i,q),(i',p) \in M$. Since at least one of $u=(i,p)$ and $v=(i',q)$ also belongs to $M$, at least three of the four vertices $(i,p),(i,q),(i',p),(i',q)$ belong to $M$. Under the isomorphism $\varphi$, these four vertices correspond to the four edges $x_iy_p$, $x_iy_q$, $x_{i'}y_p$, $x_{i'}y_q$ of a copy of $C_4$ in $K_{m,n}$. Since at least three of these four edges belong to $H_M$, the graph $H_M$ contains a $P_4$ as a subgraph. This is impossible because $H_M$ is a star forest. Hence, $u$ and $v$ are $M$-visible.

We have therefore proved that $M \subseteq V(K_m \cp K_n)$ is an outer mutual-visibility set if and only if the corresponding edge set $\varphi(M) \subseteq E(K_{m,n})$ induces a star forest. We now apply this equivalence to colorings.

Let $k=\chimuo(K_m \cp K_n)$, and let $M_1,\ldots,M_k$ be the color classes of an outer mutual-visibility $k$-coloring of $K_m \cp K_n$. Since the sets $M_1,\ldots,M_k$ form a partition of $V(K_m \cp K_n)$ and $\varphi$ is a bijection, the edge sets $E(H_{M_1}),\ldots,E(H_{M_k})$ form a partition of $E(K_{m,n})$. By the equivalence proved above, each $H_{M_i}$ is a star forest. Therefore,
$$\sa(K_{m,n}) \le k=\chimuo(K_m \cp K_n).$$
For the reverse inequality, let $k=\sa(K_{m,n})$. Then $E(K_{m,n})$ can be partitioned into $k$ star forests $H_1,\ldots,H_k$. For every $\ell \in [k]$, let
$$M_\ell=\{(i,j) \in V(K_m \cp K_n) : \varphi((i,j)) \in E(H_\ell)\}.$$
Since the edge sets $E(H_1),\ldots,E(H_k)$ form a partition of $E(K_{m,n})$ and $\varphi$ is a bijection, the sets $M_1,\ldots,M_k$ form a partition of $V(K_m\cp K_n)$. Moreover, each $H_\ell$ is a star forest, and hence each $M_\ell$ is an outer mutual-visibility set. Assigning color $\ell$ to the vertices of $M_\ell$ for every $\ell \in [k]$ gives an outer mutual-visibility $k$-coloring of $K_m \cp K_n$. Consequently,
$$\chimuo(K_m \cp K_n) \le k=\sa(K_{m,n}).$$
Combining the two inequalities yields $\chimuo(K_m\cp K_n)=\sa(K_{m,n})$, as required.
\end{proof}

\medskip

Theorem~\ref{th:hamm_chimuo} shows that determining $\chimuo(K_m \cp K_n)$ is closely related to a classical graph decomposition problem. However, in contrast to the total and dual mutual-visibility chromatic number, this does not lead to a simple closed formula, since the exact value of $\sa(K_{m,n})$ is not known for arbitrary $m$ and $n$. Nevertheless, exact values are known for some families of complete bipartite graphs. In particular, Akiyama et al.\ \cite{AK} proved that
$$\sa(K_{n,n})=\left\lceil\frac{n}{2}\right\rceil+2$$
for $n \ge 7$. Together with Theorem~\ref{th:hamm_chimuo}, this gives an exact value for the outer mutual-visibility chromatic number of $K_n\cp K_n$.

\begin{corollary}\label{cor:hamm_chimuo_equal}
For every $n \ge 7$, $\chimuo(K_n \cp K_n)=\left\lceil\frac{n}{2}\right\rceil+2$.
\end{corollary}

Note that in the case $m=n \ge 7$, the lower bound in
\eqref{eq:chimuo-m=n-lower-bound} is not sharp, but it differs from the exact value by one.

\medskip

We next consider the case when $n$ is fixed, while $m$ is sufficiently large. In this case, the lower bound in \eqref{eq:chimuo-lower-bound}, together with Theorem~\ref{th:hamm_chimuo}, can also be used in the
opposite direction. It turns out that $\sa(K_{m,n})=n$ when $m$ is sufficiently larger than $n$.

\begin{corollary}\label{cor:sa-upper-bound}
Let $n \ge 3$ and $m \ge n^2-3n+3$. Then $\sa(K_{m,n})=n$.
\end{corollary}

\begin{proof}
By Theorem~\ref{th:hamm_chimuo} and \eqref{eq:chimuo-lower-bound}, we have
$$\sa(K_{m,n})=\chimuo(K_m\cp K_n) \ge \left\lceil\frac{mn}{m+n-2}\right\rceil.$$
We first determine how large $m$ has to be in order that the right-hand side is at least $n$. Thus,
\begin{align*}
\frac{mn}{m+n-2} & > n-1,\\
mn & > (n-1)(m+n-2),\\
mn & > mn-m+(n-1)(n-2),\\
m & > (n-1)(n-2),\\
m & \ge (n-1)(n-2)+1\\
m & \ge n^2-3n+3.
\end{align*}
Hence, $\sa(K_{m,n}) \ge n$ for $m \ge n^2-3n+3$ (and $n \ge 3$). On the other hand, the edges of $K_{m,n}$ can be partitioned into $n$ stars by taking all edges incident with each vertex of the partite set of cardinality $n$. Hence, $\sa(K_{m,n})\le n$. Combining both inequalities yields
$\sa(K_{m,n})=n$.
\end{proof}

\medskip

It is interesting to compare Corollary~\ref{cor:sa-upper-bound} with a recent result of Lew \cite{Lew}. In \cite[Proposition~3.4]{Lew}, it was proved that $\sa(K_{n,m})=n$ whenever $m\ge n^2-2n+2$. Corollary~\ref{cor:sa-upper-bound} improves this sufficient condition to $m\ge n^2-3n+3$.

\section{Strong grids}\label{sec:grids}

As usual, the vertex of $P_r\boxtimes P_t$ shall be represented as $[r]\times [t]$, where two vertices $(i,j),(i',j')$ of $P_r\boxtimes P_t$ are adjacent whenever (i) $i=i'$ and $j=j'-1$, or (ii) $i=i'-1$ and $j=j'$, or (iii) $i=i'-1$ and $j=j'-1$. First, note that from some related results appearing in \cite{CDKY} and \cite{CDDH}, it is readily deduced that if $r,t\ge 3$ are integers, then
\begin{equation}\label{eq:all-mus}
\mu(P_r\boxtimes P_t)=\mut(P_r\boxtimes P_t)=\mud(P_r\boxtimes P_t)=\muo(P_r\boxtimes P_t)=2r+2t-4,
\end{equation}
where the set of vertices $(\{1,r\}\times [t])\cup ([r]\times \{1,t\})$ is the one fulfilling the conditions to represent a (resp.) corresponding mutual-visibility of largest cardinality. On the other hand, the mutual-visibility chromatic number of the strong product of paths was studied in \cite{BPSY}, where the authors proved that
for any two integers $t,r$ with $\min\{t,r\}\geq2$ it holds
\begin{equation}\label{eq:chimu-strong-grid}
\chimu(P_r\boxtimes P_t)=\left\{\begin{array}{ll}
1; & \mbox{if $r=t=2$}, \\[0.15cm]
2; & \mbox{if $r\neq t$ and $\min\{r,t\}=2$}, \\[0.15cm]
\min\big{\{}\left\lceil\frac{r}{2}\right\rceil,\left\lceil\frac{t}{2}\right\rceil\big{\}}; & \mbox{otherwise}.
\end{array}\right.
\end{equation}
It is hence our goal to consider the other parameters of the variety for this family of graphs.

\begin{theorem}
\label{th:chimut-strong-grid}
If $r\ge t\ge 2$, then
$$\chimut(P_r\boxtimes P_t)=\left\{\begin{array}{ll}
1; & \mbox{if $r=t=2$}, \\[0.15cm]
2; & \mbox{if $t=2$ and $r\ge 3$}, \\[0.15cm]
\infty; & \mbox{otherwise}.
\end{array}\right.$$
\end{theorem}

\begin{proof}
If $r=t=2$, then $P_2\boxtimes P_2$ is the graph $K_4$, and so, $\chimut(P_2\boxtimes P_2)=1$ is straightforward. If $t\ge 3$, then clearly $P_r\boxtimes P_t$ has at least one vertex which is the center of a convex $P_3$ (notice that for instance the diagonal vertices $(1,1),(2,2),(3,3)$ of $P_r\boxtimes P_t$ are such). Thus, by Theorem \ref{th_chimut-inf}, we obtain that $\chimut(P_r\boxtimes P_t)=\infty$.

Assume next that $t=2$ and $r\ge 3$. Since $P_r\boxtimes P_t$ is not a complete graph, it is readily observed that the whole vertex set of $P_r\boxtimes P_t$ is not a total mutual-visibility set. Thus, $\chimut(P_r\boxtimes P_t)\ge 2$. On the other hand, consider the two sets $S_1=[r]\times\{1\}$ and $S_2=[r]\times\{2\}$. Due to the structure of the strong product, it follows that any two vertices of $P_r\boxtimes P_t$ are $S_1$-visible, as well as, $S_2$-visible. Thus, $S_1$ and $S_2$ and total mutual-visibility sets, and so, $\chimut(P_r\boxtimes P_t)\le 2$, which leads to the equality for this last case.
\end{proof}

\begin{theorem}
If $r\ge t\ge 2$, then
$$\chimuo(P_r\boxtimes P_t)=\left\{\begin{array}{ll}
1; & \mbox{if $r=t=2$}, \\[0.15cm]
2; & \mbox{if $t=2$ and $r\ge 3$}, \\[0.15cm]
t-1; & \mbox{otherwise}.
\end{array}\right.$$
\end{theorem}

\begin{proof}
If $r=t=2$, then the conclusion is straightforward since $P_2\boxtimes P_2$ is the graph $K_4$. If $t=2$ and $r\ge 3$, then the conclusion follows from \eqref{eq:chimu-o-t}, \eqref{eq:chimu-strong-grid} and Theorem \ref{th:chimut-strong-grid}. Hence, assume $r\ge t\ge 3$.

Let $P=\{P_1\dots,P_k\}$ be an outer mutual-visibility $k$-coloring of $P_r\boxtimes P_t$ with $k=\chimuo(P_r\boxtimes P_t)$. Consider the set of vertices $D=\{(1,1),(2,2),\dots,(t,t)\}$ which induce a convex $P_t$ in $P_r\boxtimes P_t$. Now, for each $i\in [k]$, let $D_i=D\cap P_i$. Notice first that $|D_i|\le 2$ for every $i\in [k]$. That is, if $|D_\ell|>2$ for some $\ell\in [k]$, then there will be two vertices in $P_\ell$ which are not $P_\ell$-visible, since the vertices in $D$ induce a convex $P_t$. Assume next there is a $P_j\in P$ for which $|D_j|=2$ and let $D_j=\{(i,i),(i',i')\}$ with $i<i'$. If $i'<t$, then the vertices $(i,i)$ and $(t,t)$ are not $P_j$-visible, which is not possible. Similarly, if $i>1$, then the vertices $(1,1)$ and $(i',i')$ are not $P_j$-visible, which is again not possible. Consequently, we deduce that the only suitable possibility for the existence of the set $P_j$ is that $D_j=\{(1,1),(t,t)\}$. Therefore, for every other set $P_{j'}\in P$, with $j'\ne j$, it follows that $|D_{j'}|\le 1$, namely, each vertex of $D$, except possibly $(1,1)$ and $(t,t)$, belongs to different sets in $P$. Having these arguments in mind, and by the pigeon hole principle, it must happen that $\chimuo(P_r\boxtimes P_t)=k\ge t-1$.

On the other hand, let $P'=\{[r]\times \{1,t\},[r]\times \{2\},\dots, [r]\times \{t-1\}\}$ be a partition of $V(P_r\boxtimes P_t)$. It can be readily observed that each set of $P'$ forms an outer mutual-visibility set of $P_r\boxtimes P_t$, and so, $P'$ is an outer mutual-visibility $(t-1)$-coloring of $P_r\boxtimes P_t$. Consequently, we deduce the desired equality $\chimuo(P_r\boxtimes P_t)= t-1$.
\end{proof}

\begin{theorem}
If $r\ge t\ge 2$, then
$$\chimud(P_r\boxtimes P_t)=\left\{\begin{array}{ll}
1; & \mbox{if $r=t=2$}, \\[0.15cm]
2; & \mbox{if $3\le r\le 4$ and $t=3$}, \\[0.15cm]
t; & \mbox{if $2\le t\le 4$ and $(r,t)\notin \{(2,2),(3,3),(4,3)\}$}, \\[0.15cm]
\infty; & \mbox{otherwise}.
\end{array}\right.$$
\end{theorem}

\begin{proof}
If $r=t=2$, then the result is directly obtained. If $r=t=3$, then hand-made calculations show that the sets $A=\{(1,1),(1,3),(2,2)\}$ and $B=V(P_3\boxtimes P_3)\setminus A$ form a dual mutual-visibility coloring of $P_3\boxtimes P_3$. Also, we can similarly check that $A=\{(1,1),(1,2),(1,3),(2,1),(2,3),(3,2)\}$ and again $B=V(P_4\boxtimes P_3)\setminus A$ is a dual mutual-visibility coloring of $P_4\boxtimes P_3$. Thus, the equality $\chimud(P_3\boxtimes P_3)=\chimud(P_4\boxtimes P_3)=2$ follows. Assume now that $2\le t\le 4$ with $(r,t)\notin \{(2,2),(3,3),(4,3)\}$. Let $P$ be a partition of $V(P_r\boxtimes P_t)$ given as follows.
\begin{itemize}
  \item If $t=2$, then $P=\{[r]\times\{1\},[r]\times\{2\}\}$.
  \item If $t=3$, then $P=\{P_1,P_2,P_3\}$ where
  \begin{align*}
    P_1 & = \{i\in [r]\,:\,i\equiv 1\pmod 2\}\times\{1\}\cup \{i\in [r]\,:\,i\equiv 0\pmod 2\}\times\{2\}\}\\
    P_2 & = \{i\in [r]\,:\,i\equiv 0\pmod 2\}\times\{2\}\cup \{i\in [r]\,:\,i\equiv 1\pmod 2\}\times\{1\}\};\\
    P_3 & = [r]\times\{3\}.
  \end{align*}
  \item If $t=4$, then $P=\{P_1,P_2,P_3,P_4\}$ where
  \begin{align*}
    P_1 & = \{i\in [r]\,:\,i\equiv 1\pmod 2\}\times\{1\}\cup \{i\in [r]\,:\,i\equiv 0\pmod 2\}\times\{2\}\}\\
    P_2 & = \{i\in [r]\,:\,i\equiv 0\pmod 2\}\times\{2\}\cup \{i\in [r]\,:\,i\equiv 1\pmod 2\}\times\{1\}\};\\
    P_3 & = \{i\in [r]\,:\,i\equiv 1\pmod 2\}\times\{3\}\cup \{i\in [r]\,:\,i\equiv 0\pmod 2\}\times\{4\}\}\\
    P_4 & = \{i\in [r]\,:\,i\equiv 0\pmod 2\}\times\{4\}\cup \{i\in [r]\,:\,i\equiv 1\pmod 2\}\times\{3\}\}.
  \end{align*}
\end{itemize}
In order to exemplify some of the sets described above, a construction for the case $t=4$ is drawn in Figure \ref{fig:strong-grid}.
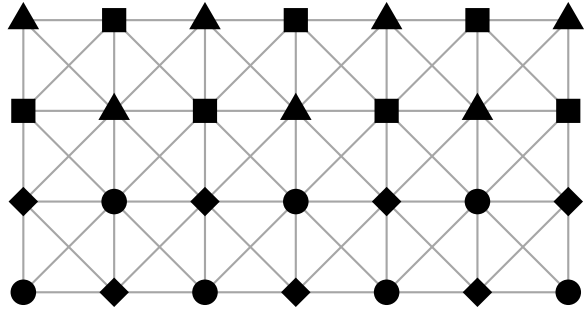
\begin{figure}
  \centering
\begin{tikzpicture}[
    scale=1.2,
    % Common node base: all shapes filled black with bold outlines
    nodeBase/.style={draw=black, ultra thick, fill=black, inner sep=0pt},
    % Distinct shapes for each partition (all filled black)
    styleP1/.style={nodeBase, circle, minimum size=8pt},
    styleP2/.style={nodeBase, diamond, minimum size=9.5pt},
    styleP3/.style={nodeBase, rectangle, minimum size=7.5pt},
    styleP4/.style={nodeBase, regular polygon, regular polygon sides=3, minimum size=10.5pt},
    % Edge style
    edge/.style={thick, gray!70}
]
  % --- DRAW EDGES ---
  % Strong product graph P_7 \boxtimes P_4
  \foreach \i in {1,...,7} {
    \foreach \j in {1,...,4} {
      % Horizontal edges
      \ifnum\i<7
        \draw[edge] (\i,\j) -- (\i+1,\j);
      \fi
      % Vertical edges
      \ifnum\j<4
        \draw[edge] (\i,\j) -- (\i,\j+1);
      \fi
      % Diagonal UP-RIGHT edges
      \ifnum\i<7
        \ifnum\j<4
          \draw[edge] (\i,\j) -- (\i+1,\j+1);
        \fi
      \fi
      % Diagonal DOWN-RIGHT edges
      \ifnum\i<7
        \ifnum\j>1
          \draw[edge] (\i,\j) -- (\i+1,\j-1);
        \fi
      \fi
    }
  }
  % --- DRAW VERTICES WITH PARTITION SHAPES ---
  \foreach \i in {1,...,7} {
    \foreach \j in {1,...,4} {

      \pgfmathsetmacro{\isOdd}{mod(\i, 2) == 1 ? 1 : 0}

      % Row y=1
      \ifnum\j=1
        \ifnum\isOdd=1
          \node[styleP1] at (\i,\j) {};
        \else
          \node[styleP2] at (\i,\j) {};
        \fi
      \fi

      % Row y=2
      \ifnum\j=2
        \ifnum\isOdd=1
          \node[styleP2] at (\i,\j) {};
        \else
          \node[styleP1] at (\i,\j) {};
        \fi
      \fi

      % Row y=3
      \ifnum\j=3
        \ifnum\isOdd=1
          \node[styleP3] at (\i,\j) {};
        \else
          \node[styleP4] at (\i,\j) {};
        \fi
      \fi

      % Row y=4
      \ifnum\j=4
        \ifnum\isOdd=1
          \node[styleP4] at (\i,\j) {};
        \else
          \node[styleP3] at (\i,\j) {};
        \fi
      \fi
    }
  }
\end{tikzpicture}
  \caption{The graph $P_7\boxtimes P_4$ is drawn and the partition $P=\{P_1,P_2,P_3,P_4\}$ is represented as follows: $P_1$ appears in black circles, $P_2$ appears in black diamonds, $P_3$ appears in black squares, and $P_4$ appears in black triangles.}\label{fig:strong-grid}
\end{figure}
Notice that the sets $P_1$ and $P_2$ (from the case $t=2$), as well as, the set $P_3$ (from the case $t=3$), can be easily checked to be dual mutual-visibility sets of the corresponding graph $P_r\boxtimes P_t$ (based on the structural properties of the strong product of graphs). We next show that $P_1$ is a dual mutual-visibility set of  $P_r\boxtimes P_3$. In this sense, let $(i,j),(i',j')\in V(P_r\boxtimes P_3)$ such that either $(i,j),(i',j')\in P_1$ or $(i,j),(i',j')\notin P_1$.

If $(i,j),(i',j')\in P_1$, then we can (w.l.o.g.) assume that $i<i'$ If $i=i'-1$, then $(i,j),(i',j')$ are adjacent and clearly $P_1$-visible. Assume then $i'-i>1$. Notice that for each $\ell\in [r]$, it holds $P_1\cap \{(\ell,1),(\ell,2)\}$ has exactly one vertex. Hence, we can consider the shortest path $Q=(i,j)(i+1,j_1)\dots(i'-1,j_{j'-j})(i',j')$ such that for every $q\in [j'-j]$ we have $(i+q,j_q)\notin P_1$. Thus, also in this case, $(i,j),(i',j')$ are $P_1$-visible.

Assume now that $(i,j),(i',j')\notin P_1$. If $i=i'$, then assume $j<j'$ and the shortest path $(i,j)(i,j+1)\dots(i,j'-1)(i,j')$ might have at most one vertex $(i,\ell)$ in $P_1$ that can be avoided (it is the case) by using one of its neighbors $(i+1,\ell)$ or $(i-1,\ell)$. Notice that at least one of these two vertices exists, and it is not in $P_1$. Hence, let $i<i'$. If $j,j'\in \{1,2\}$, then a similar shortest path to the previously used $Q$ shows that $(i,j),(i',j')$ are $P_1$-visible. If $j,j'\in \{3,4\}$, then conclusion is readily obtained since $P_1\cap ([r]\times \{3,4\})=\emptyset$. If $j\in \{1,2\}$ and $j'\in \{3,4\}$, then the vertex $(i+1,j+1)\notin P_1$ and every shortest path $Q'$ between $(i+1,j+1)$ and $(i',j')$ does not include any vertex from $P_1$. Thus, the shortest path $Q'$ together with the vertex $(i,j)$ allows to claim that also $(i,j),(i',j')$ are $P_1$-visible in this case. Any other remaining possibility for the vertices $(i,j),(i',j')$ are similar, analogous or symmetrical to the ones described above. Thus, in every case, $(i,j),(i',j')\in P_1$ or $(i,j),(i',j')\notin P_1$, are $P_1$-visible, and so, $P_1$ is a dual mutual-visibility set of  $P_r\boxtimes P_3$ as claimed.

By symmetry to the arguments above, $P_2$ is also a dual mutual-visibility set of  $P_r\boxtimes P_3$. In addition, similar techniques can be used to show also that the sets $P_1,P_2,P_3,P_4$ are dual mutual-visibility sets of $P_r\boxtimes P_4$. Therefore, altogether we conclude that $\chimud(P_r\boxtimes P_t)\le t$ for every $t\in \{2,3,4\}$. Now, in order to prove that $\chimud(P_r\boxtimes P_t)\ge t$, we proceed separately for each value of $t$.

\medskip
\noindent
\emph{Case 1:} $t=2$. The result follows directly from the fact that $P_r\boxtimes P_2$ is not a complete graph for $r\ge 3$, which means its whole vertex set is not a dual mutual-visibility set.

\medskip
\noindent
\emph{Case 2:} $t=3$. Suppose that $\chimud(P_r\boxtimes P_3)=2$ (notice that it cannot be $\chimud(P_r\boxtimes P_3)=1$) and let $\{P_1,P_2\}$ be a dual mutual-visibility $2$-coloring of $P_r\boxtimes P_3$. Since $r\ge 5$, we consider the two diagonal sets $D_1=\{(1,1),(2,2),(3,3)\}$ and $D_2=\{(2,1),(3,2),(4,3)\}$, each inducing a convex $P_3$ in $P_r\boxtimes P_t$. In this sense, for any $i,j\in [2]$, it must happen $D_i\cap P_j\ne \emptyset$. If $D_1\cap P_1=\{(2,2)\}$ (or resp. $D_1\cap P_2=\{(2,2)\}$), then the two vertices $(1,1),(3,3)\notin P_1$ (resp. $(1,1),(3,3)\notin P_2$) are not $P_1$-visible (resp. not $P_2$-visible), which is not possible. Thus, we first w.l.o.g. assume $D_1\cap P_1=\{(1,1),(2,2)\}$, and so, $D_1\cap P_2=\{(3,3)\}$. Similarly, it cannot happen $D_2\cap P_1=\{(3,2)\}$ neither $D_2\cap P_2=\{(3,2)\}$. Thus, either $D_2\cap P_1=\{(2,1),(3,2)\}$, or $D_2\cap P_1=\{(3,2),(4,3)\}$, or $D_2\cap P_1=\{(2,1)\}$, or $D_2\cap P_1=\{(4,3)\}$. We next analyze all four possibilities.
\begin{itemize}
  \item If $D_2\cap P_1=\{(2,1),(3,2)\}$, then the two vertices $(1,1),(3,2)\in P_1$ are not $P_1$-visible, which is not possible.
  \item If $D_2\cap P_1=\{(3,2),(4,3)\}$, then the two vertices $(2,1),(3,3)\notin P_1$ are not $P_1$-visible, which is again not possible.
  \item Assume $D_2\cap P_1=\{(2,1)\}$. Hence, $D_2\cap P_2=\{(3,2),(4,3)\}$. If $(3,1)\in P_1$, then $(1,1),(3,1)\in P_1$ are not $P_1$-visible, which is not possible. Thus, $(3,1)\in P_2$, which means that the three vertical vertices $(3,1),(3,2),(3,3)\in P_2$. Hence, in order to keep the dual mutual-visibility property with respect to $P_2$, every for vertex $(i,j)$ with $i\ge 4$ it holds $(i,j)\in P_2$. However, since $r\ge 5$, we have that $(3,j),(5,j')\in P_2$ are not $P_2$-visible, for any $j,j'\in [3]$. 
  \item Assume $D_2\cap P_1=\{(4,3)\}$. Hence, $D_2\cap P_2=\{(2,1),(3,2)\}$. However, in such case, we immediately obtain that $(2,2),(4,3)\in P_1$ are not $P_2$-visible, which is not possible.
\end{itemize}
It remains now to consider the case $D_1\cap P_1=\{(2,2),(3,3)\}$, which consequently means $D_1\cap P_2=\{(1,1)\}$. Moreover, by similar arguments as above, we also have that either $D_2\cap P_1=\{(2,1),(3,2)\}$, or $D_2\cap P_1=\{(3,2),(4,3)\}$, or $D_2\cap P_1=\{(2,1)\}$, or $D_2\cap P_1=\{(4,3)\}$ and consider all four possibilities.
\begin{itemize}
  \item If $D_2\cap P_1=\{(2,1),(3,2)\}$, then the two vertices $(2,1),(3,3)\in P_1$ are not $P_1$-visible, which is not possible.
  \item If $D_2\cap P_1=\{(3,2),(4,3)\}$, then the two vertices $(2,2),(4,3)\in P_1$ are not $P_1$-visible, which is again not possible.
  \item Assume $D_2\cap P_1=\{(2,1)\}$. Hence, $D_2\cap P_2=\{(3,2),(4,3)\}$. In such situation, the vertices $(1,1),(3,2)\in P_2$ are not $P_1$-visible, a contradiction. 
  \item Assume $D_2\cap P_1=\{(4,3)\}$. Hence, $D_2\cap P_2=\{(2,1),(3,2)\}$. If $(3,1)\in P_2$, then in order to keep the mutual-visibility property for the vertex $(2,1)\in P_2$ with respect to $P_2$, every vertex $(i,j)$ with $i\ge 4$ holds that $(i,j)\in P_1$. However, since $r\ge 5$, the vertices $(3,3)\in P_1$ and $(5,j')\in P_1$ with $j'\in [3]$ are not $P_1$-visible, which is not possible. Thus, it must happen $(3,1)\in P_1$. If $(4,2)\in P_2$, then $(3,1),(4,3)\in P_1$ are not $P_2$-visible, a contradiction. Hence, $(4,2)\in P_1$, which together with the fact that $(3,1)\in P_1$ leads to $(5,3)\in P_2$. But then, $(5,3),(3,2)\in P_2$ are not $P_1$-visible, which is a final contradiction.
\end{itemize}
Therefore, altogether, the arguments above lead to confirm that a dual mutual-visibility $2$-coloring of $P_r\boxtimes P_3$ is not possible if $r\ge 5$, which means that $\chimud(P_r\boxtimes P_3)\ge 3$ in such case, and the desired equality holds.

\medskip
\noindent
\emph{Case 3:} $t=4$. Suppose first that $\chimud(P_r\boxtimes P_4)=2$ and let $\{P_1,P_2\}$ be a dual mutual-visibility $2$-coloring of $P_r\boxtimes P_4$. Consider the diagonal set $D_1=\{(1,1),(2,2),(3,3),(4,4)\}$ that induces a convex $P_4$. First, notice that $|D_1\cap P_1|\le 2$ and $|D_1\cap P_2|\le 2$. Moreover, it cannot happen $D_1\cap P_1\subseteq \{(2,2),(3,3)\}$ neither $D_1\cap P_2\subseteq \{(2,2),(3,3)\}$. For otherwise, if $D_1\cap P_1\subseteq \{(2,2),(3,3)\}$ or $D_1\cap P_2\subseteq \{(2,2),(3,3)\}$, then $(1,1)$ and $(4,4)$ are not $P_1$-visible or not $P_2$-visible, respectively. Thus, w.l.o.g. $D_1\cap P_1 = \{(1,1),(2,2)\}$ and $D_1\cap P_2=\{(3,3),(4,4)\}$. On the other hand, by the symmetry of $P_r\boxtimes P_4$, we may also assume that $D'_1\cap P_1 = \{(1,4),(2,3)\}$ and that $D'_1\cap P_2=\{(3,2),(4,1)\}$, where $D'_1=\{(1,4),(2,3),(3,2),(4,1)\}$ induces another convex $P_4$. Consider now the vertex $(2,1)$. If $(2,1)\in P_1$, then $(1,1),(3,2)\in P_1$ are not $P_1$-visible. On the contrary, if $(2,1)\in P_2$, then $(2,1),(3,3)\notin P_1$ are not $P_1$-visible. Both situations lead to a contradiction. Now, if we assume $D'_1\cap P_1 = \{(3,2),(4,1)\}$ and $D'_1\cap P_2\subseteq \{(1,4),(2,3)\}$, then we obtain similar contradictions as above, by considering the vertex $(1,2)$ instead of $(2,1)$. Therefore, we conclude that $\chimud(P_r\boxtimes P_4)=2$ is not possible.

Suppose next that $\chimud(P_r\boxtimes P_4)=3$ and let $\{P_1,P_2,P_3\}$ be a dual mutual-visibility $3$-coloring of $P_r\boxtimes P_4$. As in the situation above (when we supposed $\chimud(P_r\boxtimes P_4)=2$), we deduce that it cannot happen $D_1\cap P_1\subseteq \{(2,2),(3,3)\}$ neither $D_1\cap P_2\subseteq \{(2,2),(3,3)\}$, neither $D_1\cap P_3\subseteq \{(2,2),(3,3)\}$. Thus, w.l.o.g. $D_1\cap P_1 = \{(1,1),(2,2)\}$ and $D_1\cap P_2=\{(3,3),(4,4)\}$. We again consider the diagonal set $D'_1=\{(1,4),(2,3),(3,2),(4,1)\}$, which induces another convex $P_4$. Hence, for $D'$, if $D'_1\cap P_1 = \{(1,4),(2,3)\}$ and $D'_1\cap P_2=\{(3,2),(4,1)\}$ or viceversa, then we proceed as above, and obtain a contradiction. Thus, $D'\cap P_3\ne\emptyset$, and indeed, either $D'_1\cap P_3 = \{(1,4),(2,3)\}$ or $D'_1\cap P_3=\{(3,2),(4,1)\}$. Assume first $D'_1\cap P_3 = \{(1,4),(2,3)\}$. Hence, either $D'_1\cap P_1 = \{(3,2),(4,1)\}$ or $D'_1\cap P_2 = \{(3,2),(4,1)\}$.

Assume $D'_1\cap P_1 = \{(3,2),(4,1)\}$, and consider the vertex $(2,1)$. If $(2,1)\in P_1$, then the two vertices $(1,1),(3,2)\in P_1$ are not $P_1$-visible, which is not possible. If $(2,1)\notin P_1$, then the two vertices $(2,1),(3,3)\notin P_1$ are not $P_1$-visible, which is again not possible. If $D'_1\cap P_2 = \{(3,2),(4,1)\}$, then we obtain similar contradictions, by considering the vertex $(4,2)$ instead of $(2,1)$. Moreover, if instead of $D'_1\cap P_3 = \{(1,4),(2,3)\}$ (assumed above), we suppose that $D'_1\cap P_3 = \{(3,2),(4,1)\}$, then by using analogous arguments, we obtain similar contradictions that allow to finally show that $\chimud(P_r\boxtimes P_4)=3$ is not possible. Therefore, $\chimud(P_r\boxtimes P_4)\ge 4$, which completes the proof of this case, and the situations $2\le t\le 4$.

\medskip
Finally, we consider that $t\ge 5$. In such a setting, observe that $P_r\boxtimes P_t$ contains five vertices that induce a convex $P_5$ (for example $\{(1,1),(2,2),(3,3),(4,4),(5,5)\}$). Therefore, by Theorem \ref{th_chimud-inf}, we obtain that $\chimud(P_r\boxtimes P_t)=\infty$, which completes our proof.
\end{proof}

\section{Open questions}\label{sec:conclu}

A variety of mutual-visibility chromatic parameters is introduced in this work. As a consequence of the results we obtained, the following problems might be worth to consider in future.
\begin{itemize}
    \item Characterize the graphs $G$ satisfying the equality in each of the bounds from \eqref{eq:chimus-lower-bounds}.
    \item Characterize the class of graphs $G$ for which $\chimud(G)=\infty$.
    \item Find some families of graphs where the computational complexity of the decision problems studied in Section \ref{sec:complex} will be polynomial.
    \item Study the case of higher dimensional Hamming graphs $K_n^d$ for any $n\ge 2$ and $d\ge 2$ (notice we have studied the case $d=2$). In particular, the case of hypercubes ($n=2$ and $d\ge 2$) deserves particular attention.
    \item Study the variety of our work for the strong product of $k\ge 2$ paths, and indeed, for the strong product of graphs in general.
    \item Having in mind the variety of general position problems from \cite{TK-25}, together with the existence of the chromatic version for the general position number (see \cite{CDiSHTT}), it seems to be natural to consider the study of vertex colorings for the remaining general position parameters from the variety \cite{TK-25}.
\end{itemize}

\section*{Acknowledgments}

S. \ Babu acknowledges Kerala State Council for Science,Technology and Environment (KSCSTE) for the financial support. M.\ Jakovac was supported by the Slovenian Research and Innovation Agency (ARIS) under the grants  P1-0297, N1-0285, N1-0431.  D.\ Kuziak and I.\ G.\ Yero have been partially supported by the Spanish Ministry of Science and Innovation through the grant PID2023-146643NB-I00.

%%%%%%%%%%%%%%%%%%%%%%%%%%%%%%%%%%%%%%%%%%%%%%%%


\begin{thebibliography}{99} 

\bibitem{AK} J.~Akiyama, M.~Kano, Path factors of a graph. In: Graphs and Applications. Proceedings of the first Colorado Symposium on Graph Theory (F.~Harary, J.~Maybee, eds.) (1982) 1--21.

\bibitem{AJ} J.A.~Andrews, S.~Jacobson, On a generalization of chromatic number, Congr.\ Num.\ 47 (1985) 33--45.

\bibitem{Axenovich} M.~Axenovich, D.~Liu, A note on the mutual-visibility coloring of hypercubes, arXiv:2411.12124.

\bibitem{Axenovich-pub} M.~Axenovich, D.~Liu, Visibility in hypercubes, Graphs Comb.\ 42 (2026) 38.

\bibitem{BBSL} S.~Babu, B.~Bre\v sar, B.~Samadi, A.~Lakshmanan S, Distance mutual-visibility coloring: relations with (total) domination, exact distance graphs and graph products, Quaestiones Mathematicae. In press. https://doi.org/10.2989/16073606.2026.2717535.

\bibitem{BDiSL} S.~Babu, G.~Di Stefano, A.~Lakshmanan S,
Mutual-visibility coloring of graphs. arXiv:2512.12251 (12 Feb 2026).

\bibitem{BPSY} B.~Bre\v sar, I.~Peterin, B.~Samadi, I.G.~Yero, Independent mutual-visibility coloring and related concepts, Aequationes Math.\ 100 (2026) 48.

\bibitem{BY} B.~Bre\v sar, I.G.~Yero, Lower (total) mutual-visibility in graphs, Appl.\ Math.\ Comput.\ 465 (2024) 128411.

\bibitem{BKT} Cs.~Bujt\'as, S.~Klav\v{z}ar, J.~Tian, Total mutual-visibility in Hamming graphs, Opuscula Math.\ 45 (2025) 63--78.

\bibitem{CDiSHTT} U.~Chandran S.~V., G.~Di Stefano, Haritha~S., E.~J.~Thomas, J.~Tuite,
Colouring a graph with position sets,
Ars Math.\ Contemp.\ In press. \url{https://doi.org/10.26493/1855-3974.3454.a3c}

\bibitem{CKT-survey} U.~Chandran S.~V., S.~Klav\v{z}ar, J.~Tuite,
The general position problem in graph theory: A survey.
arXiv:2501.19385v2 [math.CO] (25 Apr 2025)

\bibitem{CDDH} S.~Cicerone, G.~Di~Stefano, L.~Dro\v{z}\dj{e}k, J.~Hed\v{z}et, S.~Klav\v{z}ar, I.G.~Yero, Variety of mutual-visibility in graphs, Theoret.\ Comput.\ Sci.\ 974 (2023) 114096.

\bibitem{CDK} S.~Cicerone, G.~Di~Stefano, S.~Klav\v{z}ar, On the mutual-visibility in Cartesian products and in triangle-free graphs, Appl.\ Math.\ Comput.\ 438 (2023) 127619.

\bibitem{CDKY} S.~Cicerone, G.~Di~Stefano, S.~Klav\v{z}ar, I.G.~Yero, Mutual-visibility in strong products of graphs via total mutual-visibility, Discrete Appl.\ Math.\ 358 (2024) 136--146.

\bibitem{CDKY2} S.~Cicerone, G.~Di~Stefano, S.~Klav\v{z}ar, I.G.~Yero, Mutual-visibility problems on graphs of diameter two,
European J.\ Combin.\ 120 (2024) 103995.

\bibitem{CCW} L.J.~Cowen, R.H.~Cowen, D.R.~Woodall, Defective colorings of graphs in surfaces: Partitions into subgraphs of bounded valency, J.\ Graph Theory 10 (1986) 187--195.

\bibitem{DiS} G.~Di~Stefano, Mutual visibility in graphs, Appl.\ Math.\ Comput.\ 419 (2022) 126850.

\bibitem{GJ} M.R.~Garey, M.R.~Johnson, Computers and Intractability: A Guide to the Theory of NP-Completeness, Freeman, New York (1979).

\bibitem {HIK} R.~Hammack, W.~Imrich, S.~Klav\v{z}ar, Handbook of product graphs, Second Edition, CRC Press, Boca Raton, FL, 2011.

\bibitem{KKVY} S.~Klav\v zar, D.~Kuziak, J.C. Valenzuela-Tripodoro, I.G.~Yero, Coloring the vertices of a graph with mutual-visibility property, Open Math.\ 23(1)  (2025) 20250193.

\bibitem{KV} D.~Kor\v ze, A.~Vesel, Mutual-visibility sets in Cartesian products of paths and cycles, Results Math.\ 79 (2024) 116.

\bibitem{SRG} D.~Kuziak, M.L.~Puertas, J.A.~Rodr\'iguez-Vel\'azquez, I.G.~Yero, Strong resolving graphs: the realization and the characterization problems, Discrete Appl.\ Math.\ 236 (2018) 270--287.

\bibitem{Lew} A.~Lew, Partition density, star arboricity, and sums of Laplacian eigenvalues of graphs, J.\ Combin.\ Theory Ser.\ B 179 (2026) 71--89.

\bibitem{Kuziak} D.~Kuziak, J.A.~Rodr\'{\i}guez-Vel\'{a}zquez, Total mutual-visibility in graphs with emphasis on lexicographic and Cartesian products, Bull.\ Malays.\ Math.\ Sci.\ Soc.\ 46 (2023) 197.

\bibitem{Liu} C.-H.~Liu, Defective coloring is perfect for minors, Combinatorica 44 (2024) 467--507.

\bibitem{Pete} I. Peterin, Intervals and convex sets in strong product of graphs, Graphs Combin. 29 (2013) 705--714.

\bibitem{RKLT} D.~Roy, S.~Klav\v{z}ar, A.~Lakshmanan S, J.~Tian, Varieties of mutual-visibility and general position in Sierpi\'nski graphs, Discussiones Mathematicae Graph Theory, 46 (2026), 693--710.

\bibitem{st-74} J.G.~Stemple, Geodetic graphs of diameter two, J. Combin. Theory Ser. B, 17 (1974) 266--280.

\bibitem{TK} J.~Tian, S.~Klav\v{z}ar, Graphs with total mutual-visibility number zero and total mutual-visibility in Cartesian products, Discuss.\ Math.\ Graph Theory, 44 (2024) 1277--1291.

\bibitem{TK-25} J.~Tian, S.~Klav\v{z}ar,
Variety of general position problems in graphs, Bull.\ Malays.\ Math.\ Sci.\ Soc.\ 48 (2025) 5.

\bibitem{Yan-24} Z.~Yan, Y. Peng, Bipartite Ramsey numbers of cycles, European J.\ Combin.\ 118 (2024) 103921.

\bibitem{we} D.B.~West, Introduction to Graph Theory (Second Edition), Prentice Hall, USA, 2001.

\bibitem{schafer-1978} 
T.~J.~Schaefer, The complexity of satisfiability problems, Proceedings of the 10th Annual ACM Symposium on Theory of Computing. San Diego, California, (1978) 216--226.

\end{thebibliography}
\end{document}